\documentclass[12pt]{amsart}

\numberwithin{equation}{section}

\usepackage{amsfonts, amsmath, amssymb, amsgen, amsthm, amscd,
latexsym}
\usepackage{color}
\usepackage[all]{xy}

\usepackage{geometry}                
\usepackage{CJK}

\usepackage[svgnames]{xcolor}
\usepackage{hyperref}
\hypersetup{
    pdftitle={Explicit Realization of Hopf Cyclic Cohomology Classes of Bicrossed Product Hopf Algebras},    
    pdfauthor={Tao Yang},     
    pdfnewwindow=true,      
    colorlinks=true,       
    linkcolor=RoyalBlue,          
    citecolor=RoyalBlue,        
    filecolor=magenta,      
    urlcolor=cyan           
}

\usepackage{pgf}
\usepackage{tikz}
\usetikzlibrary{automata,matrix,arrows,decorations.pathmorphing}

\usepackage{ifxetex}

\ifxetex
\usepackage{fontspec}\defaultfontfeatures{Ligatures=TeX} 
\else
  \usepackage[T1]{fontenc}
  \usepackage[utf8]{inputenc}
\fi

\usepackage{syntonly}

\newtheorem{Theorem}{Theorem}[section]
\newtheorem{Proposition}[Theorem]{Proposition}
\newtheorem{Corollary}[Theorem]{Corollary}
\newtheorem{Lemma}[Theorem]{Lemma}
\newtheorem{Definition}[Theorem]{Definition}

\newtheorem{Note}[Theorem]{Note}

\def\Diff{\mathop{\rm Diff}\nolimits}

\def\det{\mathop{\rm det}\nolimits}

\def\log{\mathop{\rm log}\nolimits}

\def\cop{\mathop{\rm cop}\nolimits}

\newcommand{\Conv}{\mathop{\scalebox{1.5}{\raisebox{-0.2ex}{$\ast$}}}}

\newcommand{\Zb}{\mathbb{Z}}

\newcommand{\vp}{\varphi}

\newcommand{\Cb}{\mathbb{C}}

\newcommand{\Fa}{\mathfrak{a}}
\newcommand{\Fg}{\mathfrak{g}}
\newcommand{\Fh}{\mathfrak{h}}

\newcommand{\Fl}{\mathfrak{l}}

\newcommand{\Fd}{\mathfrak{d}}

\def\0b{{\bf 0}}

\def\Cb{{\mathbb C}}

\def\Rb{{\mathbb R}}

\def\Zb{{\mathbb Z}}

\newcommand{\bC}{\mathbb{C}}

\newcommand{\cH}{\mathcal{H}}
\newcommand{\cF}{\mathcal{F}}

\newcommand{\cG}{\mathcal{G}}
\newcommand{\cA}{\mathcal{A}}
\newcommand{\cR}{\mathcal{R}}
\newcommand{\cU}{\mathcal{U}}

\newcommand{\cJ}{\mathcal{J}}

\newcommand{\ttop}{{\small{\rm top}}}
\newcommand{\Bott}{{\small{\rm Bott}}}

\newcommand{\cw}{{\small{\rm c-w}}}

\def\vp{\varphi}

\def\vp{\varphi}

\def\ify{\infty}

\def\part{\partial}

\def\text{\hbox}

\def\Diff{\mathop{\rm Diff}\nolimits}

\def\acr{\vartriangleright\hspace{-4pt}\blacktriangleleft}

\def\ify{\infty}

\newcommand{\tlt}{\triangleleft}
\newcommand{\trt}{\triangleright}

\newcommand{\lag}{\langle}
\newcommand{\rag}{\rangle}

\def\cl{\hspace{-2pt}\blacktriangleright\hspace{-2.8pt} < \hspace{-2pt}}

\def\al{\hspace{-2pt}>\hspace{-2.8pt}\vartriangleleft \hspace{-2pt}}

\def\acl{\hspace{-2pt} \blacktriangleright\hspace{-2.8pt}\vartriangleleft \hspace{-2pt}}
\def\acr{\hspace{-2pt} \vartriangleright\hspace{-2.8pt}\blacktriangleleft \hspace{-2pt}}

\def\alsub{\hspace{0pt}>\hspace{-2.8pt}\vartriangleleft \hspace{0pt}}

\def\aclsub{\hspace{0pt} \blacktriangleright\hspace{-2.8pt}\vartriangleleft \hspace{0pt}}

\def\0D{\Delta^{(0)}}
\def\1D{\Delta^{(1)}}

\def\0b{{\bf 0}}
 
\def\build#1_#2^#3{\mathrel{
\mathop{\kern 0pt#1}\limits_{#2}^{#3}}}

\def\0D{\Delta^{(0)}}
\def\1D{\Delta^{(1)}}
\def\Db{\blacktriangledown}
\def\cop{\rm cop}

\def\dt{\left.\frac{d}{dt}\right|_{_{t=0}}}

\begin{document}
\author{Tao Yang}
 \address{Department of mathematics,  
The Ohio State University, 
Columbus, OH 43210, USA}
\email{yang.1204@osu.edu}

\title[EHCCBP]{Explicit Realization of Hopf Cyclic Cohomology Classes of Bicrossed Product Hopf Algebras}


\thanks{The author would like to give special thanks to Henri Moscovici.}

\begin{abstract}


We construct a Hopf action, with an invariant trace, of a bicrossed product Hopf algebra $\cH=\big( \cU(\Fg_1)   \acr  \cR(G_2) \big)^{\cop}$ constructed from a matched pair of Lie groups $G_1$ and $G_2$, on a convolution algebra $\cA=C_c^{\ify}(G_1)\rtimes G_2^{\delta}$. We give an explicit way to construct Hopf cyclic cohomology classes of our Hopf algebra $\cH$ and then realize these classes in terms of explicit representative cocycles in the cyclic cohomology of the convolution algebra $\cA$.

\end{abstract}

\maketitle

\setcounter{tocdepth}{1}
\tableofcontents

\begingroup
\let\clearpage\relax


\section{Introduction}

The geometric Hopf algebras and their cyclic cohomology first appeared as a new geometric tool in the work \cite{connes_hopf_1998} of Connes and Moscovici on the local index formula for transversely hypoelliptic operators on foliations. They constructed a canonical isomorphism \cite[Thm.~11]{connes_hopf_1998} from the Gelfand-Fuks cohomology of the Lie algebra of formal vector fields on $\Rb_n$ (\cite{gelfand_cohomology_1970}) to the Hopf cyclic cohomology of the Hopf algebra $\cH(n)$ associated to the pseudogroup of local diffeomorphisms of $\Rb_n$. In this way they proved that the local index formula computes characteristic classes of foliations (\cite{bott_characteristic_1972}). A detailed account of the above isomorphism, based on the bicrossed product decomposition of $\cH(n)$, was provided by Moscovici and Rangipour (\cite{moscovici_hopf_2011}), see the diagram from \cite[3.61]{moscovici_hopf_2011}: \\

\begin{equation*} 
\begin{tikzpicture}[description/.style={fill=white,inner sep=2pt}]
\matrix (m) [matrix of math nodes, row sep=3em, column sep=3em, 
text height=1.5ex, text depth=0.25ex]
{  C^{\bullet}_{\ttop}(\Fa_n)& C_{\cw}^{\bullet,\bullet} ( \Fg^{\ast}, \cF)  & C_{\Bott}^{\bullet}  \big( \Omega_{\bullet} ( G     ), \Gamma \big)  \\
      & C^{\bullet}(\cH(n), \Cb_{\delta})  &  C^{\bullet}(C_c^{\ify}(G)\rtimes \Gamma )  \\  };
  
\path[transform canvas={yshift=0ex},->,font=\scriptsize]
(m-1-1) edge node[above] {$ \mathcal{E} $} (m-1-2)  (m-1-2) edge node[above] {$ \Theta $} (m-1-3)  
  (m-1-3) edge node[left] {$ \Phi_{}$} (m-2-3) (m-2-2) edge  (m-1-2) ;

\path[bend left,->]
(m-1-1) edge node [above] {$ \mathcal{D} $} (m-1-3);

\end{tikzpicture} 
\end{equation*}
where $G=\Rb^{n} \rtimes GL(n)$; $\Gamma$ is a subgroup of $\text{Diff}(n)$, diffeomorphisms group of $\Rb^{n}$.
\newline
Recently, Moscovici gave an explicit description of a basis of Hopf cyclic characteristic classes of $\cH(n)$ (\cite{moscovici_geometric_2015}) in the spirit of Chern-Weil theory. He showed that the image of $\mathcal{D}$ forms a subcomplex of the Bott complex $C_{\Bott}^{\bullet}  \big( \Omega_{\bullet} (G), \Gamma \big) $ from diagram above, which is quasi-isomorphic to Hopf cyclic via a restriction of Connes' $\Phi$ map. That is  \label{diagram0}

\begin{equation*} 
\begin{tikzpicture}[description/.style={fill=white,inner sep=2pt}]
\small\matrix (m) [matrix of math nodes, row sep=2.2em, column sep=1.8em, 
text height=0.8ex, text depth=0.2ex] 
{  C^{\bullet}_{\ttop}(\Fa_n)& C_{\cw}^{\bullet,\bullet} ( \Fg^{\ast}, \cF)  & Im (\mathcal{D}) &\\
      & C^{\bullet}(\cH(n), \Cb_{\delta})  &  Im ( \lambda^{} ) &C^{\bullet}(C_c^{\ify}(G)\rtimes \text{Diff}(n)) & C^{\bullet}(C_c^{\ify}(G)\rtimes \Gamma) \\  };
  
\path[transform canvas={yshift=0ex},->,font=\scriptsize]
(m-1-1) edge node[above] {$ \mathcal{E} $} (m-1-2)  (m-1-2) edge node[above] {$ \Theta $} (m-1-3)  
  (m-1-3) edge node[left] {$ \Phi_{d}$} (m-2-3)  (m-2-2) edge node[below] {$ \cong $} node [above]{$ \lambda^{}$} (m-2-3) (m-2-4) edge   (m-2-5);

\path[ultra thick,bend left,->,transform canvas={yshift=0.5ex},font=\scriptsize]
(m-1-1) edge node [above] {$ \mathcal{D} $} (m-1-3);

\path[right hook->,font=\scriptsize]
(m-2-3) edge   (m-2-4);

\end{tikzpicture} 
\end{equation*}
This completes the description of relationship between the Lie algebra cohomology of $\Fa_n$, the Hopf cyclic cohomology of the Hopf algebra $\cH(n)$ and the cyclic cohomology of $C_c^{\ify}(G)\rtimes \text{Diff}(n)$ and $C_c^{\ify}( G )\rtimes \Gamma$. Note that the cohomology classes of the Lie algebra of formal vector fields are sophisticated, the calculation above is carried out by $\mathcal{D} $ map using simplicial connection, instead of calculating $ \mathcal{E} $ and $\Theta$ explicitly.\\
\newline
In this paper we are implementing a similar procedure for a bicrossed product Hopf algebra associated to a matched pair of Lie groups. For a matched pair of Lie groups $G_1, G_2$, with corresponding bicrossed product Lie algebra $\Fg_1 \bowtie \Fg_2$, one can also construct a bicrossed product Hopf algebra $\cR(G_2)  \acl \cU(\Fg_1) $. Rangipour and S\"{u}tl\"{u} (\cite{rangipour_van_2012}) showed that the Hopf cyclic cohomology of the Hopf algebra $\cR(G_2)  \acl \cU(\Fg_1) $ is canonically isomorphic via a van Est type quasi-isomorphism to the Lie algebra cohomology of $\Fg_1 \bowtie \Fg_2$ relative to a certain Lie subalgebra $\Fh_2$, as in the following diagram (both arrows are quasi-isomorphisms):
  
\begin{equation} \label{diagram1}
\begin{tikzpicture}[description/.style={fill=white,inner sep=2pt}]
\matrix (m) [matrix of math nodes, row sep=3em, column sep=3em, 
text height=1.5ex, text depth=0.25ex]
{   C(\Fg_1 \bowtie \Fg_2,  \Fh_2)& C^{\bullet,\bullet}_{}( \Fg_1^{\ast} , \cR_2)   \\
      & C^{\bullet}\big(\cR(G_2)  \acl \cU(\Fg_1),  ^{\sigma^{-1}}\hspace{-3pt} \Cb_{\delta}  \big)   \\  };
  
\path[transform canvas={yshift=0ex},->,font=\scriptsize]
(m-1-2) edge node[above] {$ \nu $} (m-1-1)  ;

\path[transform canvas={xshift=0ex},->,font=\scriptsize]
 (m-2-2) edge  node[right] {$ \mathcal{J} $} (m-1-2)  ;

\end{tikzpicture} 
\end{equation}

In their work, a quasi-isomorphism from the complex $C^{\bullet}\big(\cR(G_2)  \acl \cU(\Fg_1),  ^{\sigma^{-1}}\hspace{-3pt} \Cb_{\delta} \big)$ of $\cR(G_2)  \acl \cU(\Fg_1)$ with coefficient in some module comodule $ ^{\sigma^{-1}}\hspace{-3pt} \Cb_{\delta}$ to the Lie algebra cohomology complex of $(\Fg_1 \bowtie \Fg_2,\Fh_2,M)$ is constructed. However, in order to write down the Hopf cyclic classes explicitly, one would need the precise expression of the map from Lie algebra cohomology complex to the cyclic complex of Hopf algebra. 

We invert the map $\nu$, following the integration along simplices in $G/K$ map from the Lie algebra complex to the continuous group complex as in~\cite{dupont_simplicial_1976} and~\cite{van_est_group_1953}. 

Instead of inverting $\mathcal{J}$, we find a Hopf action~(\ref{injectivehomomorphism}), with an invariant trace, of it on the convolution algebra $\cA=C_c^{\ify}(G_1)\rtimes G_2^{\delta}$ ($\delta$ for the  discrete topology), as well as the convolution algebra $\cA_{\Gamma}=C_c^{\ify}(G_1)\rtimes \Gamma$ for any discrete subgroup $\Gamma$ of $ G_2$. Next we use Connes' $\Phi$ map to go directly to the cyclic cohomology of the algebra. In this case the image of $\Phi$ sits inside the range of the characteristic map from the Hopf cyclic cohomology. \\
\newline
Compared with the work of Moscovici's (\cite{moscovici_geometric_2015}), unlike the Lie algebra of formal vector fields whose cohomology classes are sophisticated, we take the advantage that for finite dimensional Lie algebra one can represent the cohomology classes by invariant forms explicitly. This way we can use $ \mathcal{E} $ and $\Theta$ directly and write down the explicit Hopf cyclic classes of our Hopf algebra $\cH=\big(\cR(G_2)  \acl \cU(\Fg_1)\big)^{\cop}$ in terms of representative cocycles on the convolution algebra, as illustrated by the following diagram: \label{diagram2}

\begin{equation} \label{diagram2}
\begin{tikzpicture}[description/.style={fill=white,inner sep=2pt}]
\small\matrix (m) [matrix of math nodes, row sep=2em, column sep=2em, 
text height=1.5ex, text depth=0.25ex]
{  C(\Fg_1 \bowtie \Fg_2,  \Fh_2)& C^{p}_{\cR}(G_2,\wedge ^{q} \Fg_1^{\ast}) & D^{p,q}    \\
       & C^{\bullet}(\cH, ^{\sigma}\Cb_{\delta})  &  Im ( \lambda^{} ) &C^{\bullet}(\cA_{}) &  C^{\bullet}(\cA_{\Gamma}) \\  };
  
\path[ultra thick,transform canvas={yshift=0ex},->,font=\scriptsize]
(m-1-1) edge node[above] {$ \mathcal{E} $} (m-1-2)  (m-1-2) edge node[above] {$ \Theta $} (m-1-3)  ;
\path[transform canvas={yshift=0ex},->,font=\scriptsize]
  (m-1-3) edge node[left] {$ \Phi_{D}$} (m-2-3)   (m-2-2) edge node[below] {$ \cong $} node [above]{$ \lambda^{}$} (m-2-3)  (m-2-4) edge   (m-2-5);

\path[bend left,->]
(m-1-1) edge node [above] {$ \mathcal{D} $} (m-1-3);

\path[right hook->]
(m-2-3) edge   (m-2-4);

\end{tikzpicture} 
\end{equation}
This amounts to an indirect way of inverting maps in diagram \ref{diagram1}, and be used to give a explicit description of Hopf cyclic classes of $\cH=\big(\cR(G_2)  \acl \cU(\Fg_1)\big)^{\cop}$. Let us known that the Hopf algebra $\big(\cR(G_2)  \acl \cU(\Fg_1)\big)^{\cop}$ actually acts on $\cA_{\Gamma}=C_c^{\ify}(G_1)\rtimes \Gamma$ as well for any discrete subgroup $\Gamma \subset G_2$, so we can complete the diagram with the very bottom right corner. \\
\\
The paper is organized as follows. In \autoref{sec:preliminaries} we introduce some background material. 
In \autoref{sec:action} we construct the bicrossed product Hopf algebra 
\newline
$\cR(G_2)  \acl \cU(\Fg_1) $. We show that the induced morphism on (Theorem~\ref{injectivehomomorphism}) Hopf action of $\big(\cR(G_2)  \acl \cU(\Fg_1)\big)^{\cop}$ on a convolution algebra $\cA=C_c^{\ify}(G_1)\rtimes G_2^{\delta}$.

We introduce differentiable and representative cohomologies of action groupoid and the complexes that calculate them in \autoref{sec:groupoid}. We also explain the construction of $D^{p,q} $ as a smaller subcomplex of the previous complexes. At the end of \autoref{sec:groupoid} we give an account of Connes' $\Phi$ map in \autoref{sec:groupoid}, whose restriction $\Phi_D$ on $D^{p,q} $ is going to be the quasi-isomorphism we need.

At the beginning of \autoref{sec:cyclic} we review the maps appear in the work of Rangipour and S\"{u}tl\"{u} (\cite{rangipour_van_2012}) as listed in diagram~\ref{diagram1} under our setting ($^{\sigma^{-1}}\hspace{-3pt}M_{\delta}=^{\sigma^{-1}}\hspace{-3pt}\Cb_{\delta}$). Then we describe the maps $\mathcal{E}$, $\mathcal{D}$, $\Theta$ in the top row of the diagram above and show that all of them are quasi-isomorphisms in the second section. At the end of \autoref{sec:cyclic} we connect the maps from \autoref{sec:action} to \autoref{sec:cyclic} and prove theorem~\ref{mthm1} that gives the same result of Rangipour and S\"{u}tl\"{u} (\cite[Thm.~4.10, Cor.~4.11]{rangipour_van_2012}) but from a different direction with explicit formulas. 

Finally we give an example calculation of transition from Lie algebra cohomology classes as explicit invariant forms on a 4 dimensional Lie group to cyclic cochains on the corresponding convolution algebra $\cA$ in \autoref{sec:examples}.


\section{Preliminaries and notations}
\label{sec:preliminaries}
In this chapter we provide background material that will be needed. We first recall the definitions of matched pair of Lie groups and Lie algebras, and then the bicrossed product Hopf algebra constructions. Most of the material is taken from~\cite{kassel_quantum_1995,majid_foundations_2000,connes_noncommutative_1994}.

\subsection{Matched pair of Lie groups and Lie algebras}

First, we would like to define a matched pair of Lie groups and a matched pair of Lie algebras. The definitions are given in Takeuchi's paper \cite{takeuchi_matched_1981}, Majid's paper \cite{majid_matched_1990} and book \cite{majid_foundations_2000}: 
\begin{Definition}\cite{takeuchi_matched_1981,majid_matched_1990,majid_foundations_2000}
Two Lie groups $(G_1,G_2)$ are a matched pair if they act on each other and the left action $\triangleright$ of $G_2$ on $G_1$, the right action $\triangleleft$ of $G_1$ on $G_2$, obey the conditions:
\begin{align}
\begin{aligned}
\forall \ \psi_{1}, \psi_{2}\in G_2,\quad \vp_1, \vp_2 &\in  G_1,\qquad \psi_1 \triangleright e = e,\quad e \triangleleft \vp_1=e,\\
\psi_1 \triangleright (\vp_1 \vp_2)&=(\psi_1 \triangleright \vp_1)\big((\psi_1 \triangleleft \vp_1)\triangleright \vp_2\big),\\
(\psi_1 \psi_2) \triangleleft \vp_1&=\big(\psi_1 \triangleleft(\psi_2 \triangleright \vp_1)\big)(\psi_2 \triangleleft \vp_1).
\end{aligned}
\end{align}
\end{Definition}
One example would be a matched pair of Lie subgroups from a decomposition of a Lie group. Let $G=G_1G_2$ as sets, $G_1\cap G_2={e}$, one defines the actions to be:
\begin{align}
\psi \vp = (\psi \triangleright \vp) (\psi \triangleleft \vp), \qquad \text{for}\ \psi \in G_2 \ \text{and} \ \vp \in G_1,
\end{align}
and check that $\triangleright$ and $\triangleleft$ are matched actions.

Differentiate these Lie groups we get a infinitesimal version: matched pair of Lie algebras:
\begin{Definition}\cite{majid_matched_1990,majid_foundations_2000}
Two Lie algebras $(\Fg_1,\Fg_2)$ are a matched pair if they act on each other and the left action $\triangleright$ of $\Fg_2$ on $\Fg_1$, the right action $\triangleleft$ of $\Fg_1$ on $\Fg_2$, obey the conditions:
\begin{align}\label{lraction}
\begin{aligned}
&[X_1,X_2] \trt Y =X_1 \trt ( X_2 \trt Y ) - X_2 \trt ( X_1 \trt Y ), \\
&X \tlt [Y_1, Y_2] =(X \tlt Y_1) \tlt Y_2 - (X \tlt Y_2) \tlt Y_1 , \\
& X \trt [Y_1, Y_2] = [X \tlt Y_1, Y_2]  +  [Y_1, X \trt Y_2] +( X \tlt Y_1 )\trt Y_2 - ( X \tlt Y_2 )\trt Y_1, \\
&[X_1,X_2] \tlt Y =[X_1 \tlt Y,X_2] + [X_1,X_2 \tlt Y] + X_1 \tlt (X_2 \trt Y) - X_2 \tlt (X_1 \trt Y).
\end{aligned}
\end{align}
\end{Definition}
Similar to the group decomposition example, one would get a matched pair of Lie subalgebras from a decomposition of a Lie algebra. Let $\Fg=\Fg_1 \oplus \Fg_2$ as vector spaces, one defines the actions to be:
\begin{align}
[Y,X] = Y \trt X + Y \tlt X, \qquad \text{for}\ Y \in \Fg_2 \ \text{and} \ X \in \Fg_1,
\end{align}
and check that $\triangleright$ and $\triangleleft$ are matched actions.

One can also assemble a match pair of Lie groups or Lie algebras to one Lie group or Lie algebra. We will introduce this when needed (cf.~\ref{assemble}).

\subsection{Bicrossed product Hopf algebras}
We refer the complete argument to \cite{majid_foundations_2000}.

We will use Sweedler notation \cite{sweedler_hopf_1969} to denote a comultiplication by $\Delta (c) = c_{(1)} \otimes c_{(2)}$, omitting the summation.\\

\begin{Definition}
Let $C$ be a coalgebra. A left $C$-comodule is a pair $(N, \Delta_{N})$ where $N$ is a vector space and $\Delta_{N}) :N \to C \otimes N$ is a linear map, called the coaction of $C$ on $N$, such that the following diagrams commute:

\begin{equation*}
\begin{tikzpicture}[description/.style={fill=white,inner sep=2pt}]
\matrix (m) [matrix of math nodes, row sep=3em, column sep=3em, 
text height=1.5ex, text depth=0.25ex]
{      N & C \otimes N   &  \Cb \otimes N  & C \otimes N   \\
       C \otimes N & C \otimes C \otimes N & & N\\  };
  
\path[transform canvas={yshift=0ex},->,font=\scriptsize]
   (m-1-1) edge node[above] {$ \Delta_{N} $} (m-1-2)   (m-1-4) edge node[above] {$ \varepsilon \otimes Id $} (m-1-3)  
  (m-2-1) edge node[above] {$\Delta \otimes Id$} (m-2-2)   (m-2-4) edge node[below] {$ \cong $} (m-1-3)  (m-2-4) edge node[right] {$ \Delta_{N} $} (m-1-4)   ;

\path[transform canvas={xshift=0ex},->,font=\scriptsize]
 (m-1-1) edge node[right]{$\Delta_{N}$} (m-2-1)   (m-1-2) edge  node[right]{$Id \otimes \Delta_{N} $} (m-2-2)   ;

\end{tikzpicture} 
\end{equation*}
\end{Definition}

We would also use Sweedler notation to denote a left coaction by $\Delta_{N} (c) = c_{<-1>} \otimes c_{<0>}$. Similarly we will denote right coaction of a coalgebra on a comodule $N$ by $\Delta_{N} (c) = c_{<0>} \otimes c_{<1>}$. One shall take advantage of the signs in this notation as $c_{<0>}$ always stay in the same vector space $N$ while $c_{<-1>},c_{<1>}$ are in the coalgebra $C$, $-$,$+$ signs tell the left/right direction of the coaction. \\

More detailed information on coalgebras, Hopf algebras and their comodules can be found in \cite{kassel_quantum_1995}

\begin{Definition}
When a Hopf algebra $\cH$ acts on an algebra $A$ from the left, we say that $A$ is a left $\cH-$module algebra if  
\begin{align*}
& h  \trt (ab)  = (h_{(1)}\trt a)(h_{(2)} \trt b), \\
&h \trt (1_{A}) =\varepsilon(h ) 1_{A}.
\end{align*}
Then we can form a left cross product algebra $A \al \cH $ built on vector space $A \otimes \cH $ with the product:
\begin{align*}
&(a \al h)(b \al g)= a (h_{(1)} \trt b)\otimes   h_{(2)}  g, \qquad a,b \in A, \quad h,g \in \cH\\
&1_{A \alsub \cH } =1_{A} \otimes 1_{\cH}.
\end{align*} 
\end{Definition}

\begin{Definition}
Let $\cH$ be a Hopf algebra. Suppose a coalgebra $C$ is a right $\cH-$comodule with coaction 
\begin{align*}
\Db (c) = c_{<0>} \otimes c_{<1>},
\end{align*} 
then $C$ is a right $\cH-$comodule coalgebra if the coaction commutes with the coproduct and counit, i.e., the following conditions are satisfied: 
\begin{align*}
&c_{<0>(1)} \otimes  c_{<0>(2)}  \otimes c_{<1>} = c_{(1)<0>}  \otimes c_{(2)<0>} \otimes c_{(1)<1>}c_{(2)<1>} \\
&\varepsilon(c_{<0>}) c_{<1>} =\varepsilon(c ) 1_{\cH}.
\end{align*} 
Then we can form a right cross coproduct coalgebra $\cH \cl C$ built on vector space $\cH \otimes C$ with the coalgebra structure:
\begin{align*}
&\Delta(h \cl c)= h_{(1)} \cl c_{(1)<0>} \otimes   h_{(2)}  c_{(1)<1>} \cl c_{(2)}, \\
&\varepsilon (h \cl c )=\varepsilon(h) \varepsilon (c),
\end{align*} 
for $h \in \cH$ and $c \in C$. 
\end{Definition}

\begin{Definition}
Let $A$, $\cH$ be Hopf algebras, let $A$ be a left $\cH-$module algebra and $\cH$ be a right $A-$comodule coalgebra such that the compatibility conditions 
\begin{align*}
&\varepsilon(h \trt a)=\varepsilon(h)\varepsilon(a), \quad  \Delta(h \trt a )=h_{(1)<0>} \trt a_{(1)} \otimes h_{(1)<1>} (h_{(2)} \trt a_{(2)} ), \\
&1_{<0>} \otimes 1_{<1>} =  1 \otimes 1, \\
& (gh)_{<0>} \otimes (gh)_{<1>}=g_{(1)<0>} h_{<0>} \otimes g_{(1)<1>} (g_{(2)} \trt h_{<1>}), \\
&h_{(2)<0>} \otimes (h_{(1)}   \trt a   ) h_{(2)<1>}  =  h_{(1)<0>}  \otimes h_{(1)<1>} ( h_{(2)} \trt a),
\end{align*}
are satisfied for all $a,b \in A$ and $g,h \in \cH$. Then the algebra $A \al \cH$ and coalgebra $A \cl \cH$ form a left-right bicrossed product Hopf algebra $A \acl \cH$ with antipode
\begin{align*}
S_{\aclsub} (a \acl h)= \big(1 \acl S(h_{<0>}) \big) \big( S(a h_{<1>})\acl 1 \big)
\end{align*}
\end{Definition}

\subsection{Additional notation}
In this paper $G$ will denote a Lie group, $G^{\delta}$ means same group but with discrete topology, $H$ and $L$ are closed subgroups; $\Fg$, $\Fh$, and $\Fl$ will denote the corresponding Lie algebras. $X$,$Y$,$Z$ are vectors of Lie algebras. $\displaystyle \widetilde{X}$,$\displaystyle \widetilde{Y}$,$\displaystyle \widetilde{Z}$ means the corresponding left invariant vector field. $\cH$ is Hopf algebra, $\cA$ is convolution algebra. $\Db$ is for coaction and $\acl$ is for left-right bicrossed product. $\cU$ is universal enveloping algebra, $\cR$ is algebra of representative functions.


\section{Hopf algebra \texorpdfstring{$\cH$}{H} and its standard action}
\label{sec:action}

\subsection{Hopf algebra \texorpdfstring{$\cH$}{H}}
In this section, our goal is to associate a bicrossed product Hopf algebra and a convolution algebra on which the Hopf algebra acts to any matched pair of Lie groups. We prove theorem~\ref{injectivehomomorphism} and show that the characteristic map is faithful.\\

Given a matched pair of Lie groups $(G_1,G_2)$, we view $G_2$ as discrete group and denote as $G_2^{\delta}$. Consider the action groupoid $\cG = G_1 \rtimes G_2^{\delta} $ and its convolution algebra $C_c^{\ify}(\cG)$, which is equivalent to the cross product algebra $\cA=C_c^{\ify}(G_1)\rtimes G_2^{\delta}$ (see~\cite{renault_groupoid_1980,da_silva_geometric_1999}). The elements of $\cA$ are finite sums of symbols of the form
\begin{align*}\label{}
 fU_{\psi}^{\ast},\qquad \text{where} \quad f \in C_{c}^{\infty}(G_1),\quad U_{\psi}^{\ast}\, \text{ stands for } \widetilde{\psi^{-1}} \in G_2,
\end{align*} 
with multiplication:
\begin{align}\label{algebramultiplication}
fU_{\psi_1}^{\ast} \Conv gU_{\psi_2}^{\ast}=f (g\circ \widetilde{\psi_1}) U^\ast_{\psi_2 \psi_1} ,
\end{align} 
where $\widetilde{\psi_1} \in \Diff (G_1)$ is left action by $\psi_1$. We will also use $fU_{\psi}:= fU_{\psi^{-1}}^{\ast}$ to simplify notation afterwards. 
\newline
Denote the Lie algebra of $G_1$ by $\Fg_1$ and its basis element by $Z_i,1 \le i \le \dim (G_1)$, the corresponding left-invariant vector field on $G_1$ by $\widetilde{Z}_i$. These vector fields act on $\cA$ by
\begin{align}
\widetilde{Z}_i (fU_{\psi}^{\ast} )=\widetilde{Z}_i (f)U_{\psi}^{\ast}, \qquad \widetilde{Z}_i (f)U_{\psi}^{\ast}\vert_{\vp_0} = \dt f \big(\vp_0 \exp(tZ_i) \big) U_{\psi}^{\ast}.
\end{align}

Now we calculate $L_{\psi^\ast} \widetilde{Z}_i$ for $ \psi \in  G_2$:
\begin{align}	
\begin{aligned}	 
	\widetilde{Z}_i(f\circ \widetilde{\psi})\vert_{\vp_0} &=\dt (f\circ \widetilde{\psi}) \big(\vp_0 \exp(tZ_i) \big)\\
				&=\dt f \Big(  \psi  \trt   \big(\vp_0 \exp(tZ_i) \big) \Big)\\
				&=\dt f  \Big(( \psi \trt \vp_0)\big((\psi \tlt \vp_0)\trt \exp(tZ_i)\big) \Big)\\
				&=\dt f  \Big(( \psi \trt \vp_0)\big((\psi^{-1}\tlt (\psi \trt \vp_0))^{-1}\trt \exp(tZ_i) \big)  \Big)\\
				&=\sum_{j}  \big(\Gamma_{i}^j(\psi^{-1})\widetilde{Z}_j(f) \big)\circ \widetilde{\psi}  |_{\vp_0},
\end{aligned}	
\end{align}
or in short
\begin{align}
\label{zpsi}
\widetilde{Z}_i U_{\psi}^{\ast}  &= \sum_{j}  U_{\psi}^{\ast}  \Gamma_{i}^j(\psi^{-1})\widetilde{Z}_j ,
\end{align}
where 
\begin{align}
\label{gamma}
\sum_{j} \Gamma_{i}^j(\psi)(\vp_0)Z_j:=(\psi^{}\tlt \vp_0)^{-1} \trt Z_i , \quad \text{or}\quad
\Gamma_{i}^j(\psi)(\vp_0):=\big\lag(\psi\tlt \vp_0)^{-1} \trt Z_i,\omega_j\big\rag  ,
\end{align}
and $\{\omega_j,1 \le j \le \dim (G_1)\}$ is the dual basis and the action $\trt$ of $G_2$ on $\Fg_1$ is given by the differentiation of $\trt$ on $G_1$ at $e$.

We notice that $\Gamma$ satisfies:
\begin{align}
\label{1-gamma}
\begin{aligned}
&\Gamma_i^j (\psi_1\psi_2)= \sum_{k}  (\Gamma_{i}^{k} (\psi_1) \circ \widetilde{\psi_2} )\Gamma_{k}^{j} (\psi_2) , \\
&\Gamma_i^j (\psi^{-1}) \circ \widetilde{\psi} =(\Gamma^{-1} )_{i}^{j} (\psi).
\end{aligned}
\end{align}

From~\ref{gamma} the definition of $\Gamma_{i}^{j} $, we can see that $\Gamma_{i}^{j} $ is determined only by its value at $e$:
\begin{align}
\label{ytltx}
\Gamma_{i}^{j}  (\psi)(\vp_0)=\Gamma_{i}^{j} (\psi \tlt \vp_0)(e),
\end{align}
we will denote 
\begin{align}
\gamma_i^j(\psi):=\Gamma _{i}^{j} (\psi)(e) ,
\qquad \text{or}\qquad \label{coact}
\gamma_{i}^j(\psi):=\lag \psi^{-1} \trt Z_i,\omega_j \rag.
\end{align}
Hence similarly we can show $\gamma$ is a homomorphism:
\begin{align}
\label{one-cocycle condition}
\begin{aligned}
&\gamma_i^j(\psi_1 \psi_2)=  \gamma_i^k(\psi_1) \gamma_k^j(\psi_2) , \\
&\gamma_i^j (\psi^{-1})  =(\gamma^{-1} )_{i}^{j} (\psi)
\end{aligned}
\end{align}

This property of $\gamma_i^j$ suggests that $\gamma_i^j$ falls into a special kind of functions on $G_2$ called representative functions. We will give the definition originated from~\cite{hochschild_representations_1957}. 

\begin{Definition}\cite{hochschild_representations_1957}
Fix the base field to be $\mathbb{C}$, let $G$ be a Lie group and $\rho : G \to Aut(V)$ a finite dimensional smooth representation. Topologize $End(V)$ so that every linear functional is continuous and topologize $Aut(V)$ by the induced topology. Then, the composition of $\rho$ with a linear functional $\tau \in Aut(V)^{\ast}$ is called a representative function of $G$. Denote $\cR(G)$ the algebra of representative functions when we run over all pairs of finite dimensional representation and linear functional.
\end{Definition}
It is well known that the representative functions $\cR(G)$ form a commutative Hopf algebra:
\begin{align}
\begin{aligned}
&F_1F_2(\psi)=F_1(\psi) F_2(\psi), \quad \Delta(F)(\psi_1,\psi_2)=F(\psi_1 \psi_2) ,\\
&1(\psi)=1, \quad \varepsilon (F)= F(e), \quad S (F)(\psi)= F(\psi^{-1}).
\end{aligned}
\end{align}

From~\ref{zpsi} we can uniquely express  
\begin{align}
U_{\psi}^{\ast} 	\widetilde{Z}_i U_{\psi}^{}  = \Gamma_{i}^j(\psi)\widetilde{Z}_j ,\qquad \text{or }	\qquad\widetilde{(\psi^{-1})_{\ast} } (\widetilde{Z}_i)  = \Gamma_{i}^j(\psi)\widetilde{Z}_j.
\end{align}
This gives a left action of $G_2$ on $\Fg_1$:
\begin{align} \label{laction}
\psi \trt ({Z}_i)=\widetilde{(\psi^{})_{\ast} } (\widetilde{Z}_i)\vert _{e} =\Gamma_{i}^j(\psi^{-1})(e) {Z}_j =\gamma_{i}^j(\psi^{-1}){Z}_j =S(\gamma_i^j )(\psi) {Z}_j,
\end{align}
which can be dualized using the identification 
\begin{align} 
({Z}_j)_{<0>} ({Z}_j)_{<1>} (\psi):=\psi \trt ({Z}_i)
\end{align}
to a right coaction of $\cR(G_2)$ on $\Fg_1$:
\begin{align}
\label{rightcoaction}
\blacktriangledown (Z_i)= (Z_i )_{<0 >}  \otimes    (Z_i) _{<1 >} := Z_j \otimes S(\gamma_i^j ).
\end{align}
The transpose of the left action \ref{laction} gives a right action of $G_2$ on $\Fg_1^{\ast}$:
\begin{align}\label{raction}
 (\omega_i) \trt^t \psi=S(\gamma_i^j )(\psi) {\omega}_j,
\end{align}
dualize it we can have a left coaction of $\cR(G_2)$ on $\Fg_1^{\ast}$:
\begin{align}
\label{leftcoaction2}
\blacktriangledown (\omega_i):=  S(\gamma_i^j ) \otimes  \omega_j .
\end{align}
We also have a left action of $\Fg_1$ on $\cR(G_2)$ from differentiating the left action of $G_1$ on $C^{\ify}(G_2)$ (which comes from the right action of $G_1$ on $G_2$):
\begin{align}
\label{leftaction}
(Z_i \triangleright F )(\psi):=\dt F \big(\psi \triangleleft \exp (t Z_i) \big) ,
\end{align}
\begin{Proposition} \label{11}\cite[2.14]{rangipour_lie-hopf_2010} 
For any $Z_i \in \Fg_1$ and any $F \in \cR(G_2)$, we have $(Z_i \triangleright F ) \in \cR(G_2)$.
\end{Proposition}
Both coaction and action can be extended from $\Fg_1$ to $\cU(\Fg_1)$ and we have the following result.
\begin{Proposition}
\cite[Theorem~2.17]{rangipour_lie-hopf_2010}by the left action~\ref{leftaction} and the right coaction~\ref{rightcoaction}, the pair $(\cU(\Fg_1),\cR(G_2))$ is a matched pair of Hopf algebras. i.e., $\cR(G_2)$ is a left $\cU(\Fg_1)-$module algebra, $\cU(\Fg_1)$ is a right $\cR(G_2)-$comodule coalgebra, and they satisfy the compatibility conditions
\begin{align*}
&\varepsilon(u \triangleright F ) =\varepsilon (u) \varepsilon (F), \qquad \blacktriangledown (1)=1 \otimes 1,\\
&\Delta(u \trt F) = u_{(1)_{< 0 > }} \triangleright F_{(1)} \otimes u_{(1)_{< 1 > }}  (u_{(2)} \triangleright F_{(2)}  ),        \\
&\blacktriangledown (uv)= u_{(1)_{< 0 > }} v_{<0>}  \otimes u_{(1)_{< 1 > }} (u_{(2)}   \triangleright   v_{<1>}    ) ,\\
& u_{(2)_{< 0 > }}    \otimes  ( u_{(1)}   \triangleright F       ) u_{(2)_{< 1 > }} =u_{(1)_{< 0 > }}    \otimes  u_{(1)_{< 1 > }}   (   u_{(2)}   \trt F             ).                  
\end{align*}
\end{Proposition}
As a result of this proposition, for a matched pair of Hopf algebras $(\cU(\Fg_1),\cR(G_2))$ we can form the left-right bicrossed product Hopf algebra $\cR(G_2) \acl \cU(\Fg_1)$ in a canonical way.\\
The algebra structure is given by $\cR(G_2) \al \cU(\Fg_1)$:
\begin{align}\label{product1}
(F \acl u)(G \acl v)=(F (u_{(1)}\trt G) \acl u_{(2)}v)
\end{align}
with unit $(1 \acl 1)$ and the coalgebra structure is given by $\cR(G_2) \cl \cU(\Fg_1)$:
\begin{align}
\Delta_{\aclsub} (F \acl u)=(F_{(1)} \acl u_{(1)_{<0>}}) \otimes (F_{(2)}  u_{(1)_{<1>}} \acl u_{(2)})
\end{align}
with counit $\varepsilon (F \acl u)=\varepsilon (F) \varepsilon (u)$, and finally the antipode is given by
\begin{align}
S_{\aclsub} (F \acl u)= \big(1 \acl S(u_{<0>}) \big) \big(S(F u_{<1>})\acl 1 \big)
\end{align}

Fix an order of the basis of $\Fg_1$ and Let $I=(i_1,\dots, i_p)$, $p= \dim{G_1}$ be multi-indeces, ordered lexicographically. We then have $Z_{I}=Z_{1}^{i_1}\cdots Z_{p}^{i_p}$ as a PBW basis of $\cU(\Fg_1)$. 
Every element $\sum_{j}F_{j} \acl u_{j}$ in $\cR(G_2) \acl \cU(\Fg_1)$ can be written uniquely as
\begin{align}
\begin{aligned}\label{basis2}
\sum_{j}F_{j} \acl u_{j}=\sum_{I} F_{I} \acl Z_{I},\quad I=(i_1,\dots, i_p).
\end{aligned}
\end{align}
Therefore, $\{F\acl Z_{i}, F \in \cR(G_2) , 1 \le i \le \dim{G_1}\}$ is a generating set of $\cR(G_2) \acl \cU(\Fg_1)$.

\begin{Proposition} \label{sinvertible}
For the specific Hopf algebra $\cR(G_2) \acl \cU(\Fg_1)$ we can show that $S_{\aclsub}$ is invertible and $S_{\aclsub}^{-1}$ is given by the formula
\begin{align}
S_{\aclsub}^{-1} (F \acl Z)=\big( S( Z_{<1>}) \acl S(Z_{<0>}) \big) \big( S(F )\acl 1 \big)
\end{align}
on the generators and is extended as an anti-algebra morphism.
\end{Proposition}
\begin{proof}
The non-trivial check will be on $1 \acl Z_{i}$:
First we observe that $S(Z)=-Z$ for $Z \in  \Fg_1$ and
\begin{align}
 (-Z)_{<0 >}  \otimes    (-Z )_{<1 >} = \blacktriangledown (-Z)= -\blacktriangledown (Z)= -(Z )_{<0 >}  \otimes    (Z) _{<1 >}, \quad Z \in \Fg_1
 \end{align}
 therefore
 \begin{align}
  \big( S(Z) \big)_{<0>}\otimes \big( S(Z) \big)_{<1>}=S \big( (Z)_{<0>} \big) \otimes (Z)_{<1>}, \quad Z \in \Fg_1
\end{align}
note that the above equality is only true for $Z \in \Fg_1$ not $u \in \cU(\Fg_1)$. The above equality allow us to interchange coaction and antipode on generators.
\newline
Now we have 
\begin{align}
\begin{aligned}
S_{\aclsub}^{-1} S_{\aclsub}^{} (1 \acl Z)=&S_{\aclsub}^{-1}\Big( (1 \acl S(Z _{<0>}))(S( Z _{<1>})\acl 1)\Big)\\
=&S_{\aclsub}^{-1}\Big(S( Z _{<1>})\acl 1\Big) S_{\aclsub}^{-1}\Big( 1 \acl S(Z _{<0>})\Big)\\
=&\Big( Z _{<1>}\acl 1 \Big) \Big(S( (S(Z _{<0>}))_{<1>}) \acl S((S(Z _{<0>}))_{<0>})\Big) \\
=&\Big( Z _{<1>}\acl 1 \Big) \Big(S( (Z _{<0>})_{<1>}) \acl Z _{<0><0>}\Big) \\
=&\Big(Z _{<1>(2)} S(Z _{<1>(1)}) \acl Z _{<0>} \Big)\\
=&\varepsilon (Z_{<1>}) \acl Z _{<0>}\\
=&1 \acl Z
\end{aligned}\\
\begin{aligned}
S_{\aclsub}^{} S_{\aclsub}^{-1} (1 \acl Z)=&S_{\aclsub}^{}\big(S( Z _{<1>})\acl S(Z _{<0>})\big)\\
=& \Big(1 \acl S( (S(Z _{<0>}))_{<0>}) \Big)\Big( S((S(Z _{<0>}))_{<1>}) Z _{<1>}\acl 1 \Big)\\
=& \Big(1 \acl Z _{<0><0>} \Big)\Big( S((Z _{<0>})_{<1>}) Z _{<1>}\acl 1 \Big)\\
=& \Big(1 \acl Z _{<0>} \Big)\Big( S (Z _{<1>(1)})Z _{<1>(2)}  \acl 1 \Big)\\
=& \Big(1 \acl Z _{<0>} \Big)\Big( \varepsilon (Z_{<1>})  \acl 1 \Big)\\
=&1 \acl Z
\end{aligned}
\end{align}
\end{proof}

\begin{Definition}
Opposite coalgebra. For any coalgebra $(C,\Delta,\varepsilon)$ we set
\begin{align*}
\Delta^{\rm{op}}=\tau_{C,C} \circ \Delta
\end{align*}
where $\tau_{C,C}$ is the twist map.
Then $(C,\Delta^{\rm{op}},\varepsilon)$ is a coalgebra which we call the opposite coalgebra and denote by $C^{\cop}$.
\end{Definition}

\begin{Corollary}
The opposite coalgebra $\cH=\big(\cR(G_2) \acl \cU(\Fg_1)\big)^{\cop}$ is a Hopf algebra with coproduct $\Delta=\Delta_{\aclsub}^{\rm{op}}$ and antipode $S=S_{\aclsub}^{-1}$.
\end{Corollary}
\begin{proof}
It was proved in~\cite[Corollary III.3.5]{kassel_quantum_1995} that if the antipode of a Hopf algebra has an inverse, the opposite coalgebra is also a Hopf algebra with opposite coproduct and inverse antipode.
\end{proof}

\begin{Definition}
Let two Lie groups $(G_1,G_2)$ be a matched pair, define $\cH$ to be the opposite coalgebra $\big(\cR(G_2) \acl \cU(\Fg_1)\big)^{\cop}$, with product given by the same of $\big(\cR(G_2) \acl \cU(\Fg_1)\big)$ (\ref{product1}), coproduct given by $\Delta_{\aclsub}^{\rm{op}}$ and antipode $S_{\aclsub}^{-1}$.
\end{Definition}
From now on, we will focus on the opposite coalgebra $\cH$, which we will use to construct a Hopf action later.\\

\subsection{Standard action}
Next we would like to construct a Hopf action of $\cH$ on $\cA$:

Again take a PBW basis of $\cU(\Fg_1)$,\label{PBW} $\{Z_{I}=Z_{1}^{i_1}\cdots Z_{p}^{i_p},I=(i_1,\dots, i_p)\}$.
Every $Z_{I} \in \cU(\Fg_1)$ acts on $\cA$ by
\begin{align}
&\widetilde{Z_{I}}(fU_{\psi}^{\ast})=\widetilde{Z_{1}}^{i_1}\cdots \widetilde{Z_{p}}^{i_p}(f)U_{\psi}^{\ast}, \\
&\widetilde{Z_{I}} (f)U_{\psi}^{\ast}\vert_{\vp_0} = \underbrace{\left.\frac{d}{dt_1}\right|_{_{t_1=0}}}_{\text{$i_1$ times}} \cdots \underbrace{\left.\frac{d}{dt_p}\right|_{_{t_p=0}}}_{\text{$i_p$ times}}  f \big(\vp_0 \underbrace{\exp(t_1Z_{1})\dots}_{\text{$i_1$ times}} \cdots \underbrace{\exp(t_pZ_{p})\dots}_{\text{$i_p$ times}}  \big) U_{\psi}^{\ast},
\end{align}
on each $fU_{\psi}^{\ast}$ and extend to the finite sums.

$F \in \cR(G_2) $ act on $\cA$ by
\begin{align}
\widetilde{F}(fU_{\psi}^{\ast})(\vp)=F(\psi \tlt \vp)f(\vp)U_{\psi}^{\ast},
\end{align}
on each $fU_{\psi}^{\ast}$ and extend to the finite sums.

Therefore we would like $F  \acl Z_{I} $ to act on $\cA$ by
\begin{align}
\widetilde{F  \acl Z_{I} }(fU_{\psi}^{\ast})(\vp)=F(\psi \tlt \vp) \widetilde{Z_{1}}^{i_1}\cdots \widetilde{Z_{p}}^{i_p}(f)(\vp)U_{\psi}^{\ast},
\end{align}
or equivalently,
\begin{align}
\widetilde{F  \acl Z_{I} }(fU_{\psi}^{\ast})(\vp)=\vp \trt F(\psi^{} )\widetilde{Z_{1}}^{i_1}\cdots \widetilde{Z_{p}}^{i_p}(f)(\vp)U_{\psi}^{\ast}.
\end{align}

\begin{Proposition}
The formula
\begin{align}
\widetilde{F  \acl Z_{I} }(fU_{\psi}^{\ast})(\vp)=F(\psi \tlt \vp) \widetilde{Z_{1}}^{i_1}\cdots \widetilde{Z_{p}}^{i_p}(f)(\vp)U_{\psi}^{\ast}
\end{align}
defines a Hopf action.
\end{Proposition}

\begin{proof}
We need to check that the action is well-defined, i.e., commutes with product and coproduct.
\newline
First, we want to show it commutes with product. Because of the formula above, it is trivial that
\begin{align}
\begin{aligned}
&(\widetilde{F  \acl 1 } )(\widetilde{1  \acl Z_{I} } )=\widetilde{F  \acl Z_{I} }, \\
&(\widetilde{F  \acl 1 } )(\widetilde{G  \acl 1})=\widetilde{(F G) \acl 1 } ,\\
&(\widetilde{1  \acl Z_{I} } )(\widetilde{1  \acl Z_{J} } )=\widetilde{1  \acl Z_{I}Z_{J} }.
\end{aligned}
\end{align}
Therefore we just need to show that 
\begin{align}\label{zif}
\begin{aligned}
\widetilde{1  \acl Z_{I} } \widetilde{F \acl 1} =\widetilde{(1  \acl Z_{I} )(F \acl 1)}.
\end{aligned}
\end{align}
This is true because
\begin{align}
\begin{aligned}
&\widetilde{Z_I} \widetilde{F} (fU_{\psi}^{\ast} )(\vp)\\
=&\widetilde{Z_I} (F(\psi \tlt \vp)f(\vp)U_{\psi}^{\ast}) \\
=&F(\psi \tlt \vp)\big(\widetilde{(Z_{I})_{(1)}} f)(\vp)U_{\psi}^{\ast} + \dt F(\psi \tlt (\vp \exp t (Z_{I})_{(2)}))f(\vp)U_{\psi}^{\ast} \\
=&F(\psi \tlt \vp)\big(\widetilde{(Z_{I})_{(1)}} f)(\vp)U_{\psi}^{\ast} + \dt F((\psi \tlt \vp )(\exp t (Z_{I})_{(2)}))f(\vp)U_{\psi}^{\ast} \\
=&F(\psi \tlt \vp)\big(\widetilde{(Z_{I})_{(1)}} f)(\vp)U_{\psi}^{\ast}  + ((Z_{I})_{(2)}\trt F)(\psi \tlt \vp)f(\vp)U_{\psi}^{\ast} \\
=&\widetilde{F} \widetilde{(Z_{I})_{(1)}}  (fU_{\psi}^{\ast} )(\vp) + \widetilde{(Z_{I})_{(2)}\trt F} (fU_{\psi}^{\ast} )(\vp).
\end{aligned}
\end{align}
Next, we verify on generators that the action commutes with coproduct,
\begin{align}
\begin{aligned}
&\widetilde{F}(fU_{\psi_1}^{\ast} \Conv  gU_{\psi_2}^{\ast})(\vp)\\
=&\widetilde{F}(f (g\circ \widetilde{\psi_1}) U^\ast_{\psi_2 \psi_1})(\vp)\\
=&F(( \psi_2 \psi_1 )\tlt  \vp)f (\vp) (g\circ \widetilde{\psi_1} (\vp))U^\ast_{\psi_2 \psi_1}\\
=&F((( \psi_2 \tlt (\psi_1 \trt \vp)\psi_1\tlt \vp)f (\vp) (g\circ \widetilde{\psi_1} (\vp))U^\ast_{\psi_2 \psi_1}\\
=&F_{(1)}( \psi_2 \tlt (\psi_1 \trt \vp))  F_{(2)}(\psi_1\tlt \vp)f (\vp) (g\circ \widetilde{\psi_1} (\vp))U^\ast_{\psi_2 \psi_1}\\
=&(F_{(1)}( \psi_2 \tlt \vp)  \circ \widetilde{\psi_1} )F_{(2)}(\psi_1\tlt \vp)f (\vp) (g\circ \widetilde{\psi_1} (\vp))U^\ast_{\psi_2 \psi_1}\\
=&\widetilde{F}_{(2)}(fU_{\psi_1}^{\ast}) \Conv \widetilde{F}_{(1)}(gU_{\psi_2}^{\ast})(\vp)\\
=&\widetilde{\Delta_{\aclsub}^{\rm{op}}(F \acl 1)} (fU_{\psi_1}^{\ast} \otimes  gU_{\psi_2}^{\ast}) (\vp),
\end{aligned}
\end{align}
\begin{align}
\begin{aligned}
&\widetilde{Z_i}(fU^\ast_{\psi_1} \Conv gU^\ast_{\psi_2})\\
=&\widetilde{Z_i}(f (g\circ \widetilde{\psi_1}) U^\ast_{\psi_2 \psi_1}) \\
=&\widetilde{Z_i}(f) (g\circ \widetilde{\psi_1}) U^\ast_{\psi_2 \psi_1}+f\widetilde{Z_i} (g\circ \widetilde{\psi_1}) U^\ast_{\psi_2 \psi_1} \\
=&\widetilde{Z_i}(f) (g\circ \widetilde{\psi_1}) U^\ast_{\psi_2 \psi_1}+f\Gamma_{i}^j(\psi_1)(\widetilde{Z_j}(g)\circ \widetilde{\psi_1}) U^\ast_{\psi_2 \psi_1} \\
=&\widetilde{Z_i}(fU^\ast_{\psi_1} )\Conv gU^\ast_{\psi_2} +(\Gamma_{i}^j(\psi_1) f)U^\ast_{\psi_1} \Conv \widetilde{Z_j}(gU^\ast_{\psi_2})\\
=&\widetilde{Z_i}(fU^\ast_{\psi_1} )\Conv gU^\ast_{\psi_2} +\widetilde{\gamma_{i}^j}(fU^\ast_{\psi_1} )\Conv \widetilde{Z_j}(gU^\ast_{\psi_2})\\
=&\widetilde{\Delta_{\aclsub}^{\rm{op}}(1 \acl Z_i )} (fU_{\psi_1}^{\ast} \otimes  gU_{\psi_2}^{\ast}) (\vp).
\end{aligned}
\end{align}
\end{proof}
\begin{Note}
The formula also defines a Hopf action on $\cA_{\Gamma}=C_c^{\ify}(G_1)\rtimes \Gamma$ as well for any discrete subgroup $\Gamma < G_2^{\delta}$.
\end{Note}

In order to do Hopf cyclic theory, we want to show that this Hopf action is equipped a $\delta$-invariant $\sigma$-trace. 
We introduce a functional on the discrete crossed product $\cA=C_c^{\ify}(G_1)\rtimes G_2^{\delta}$:
\begin{align}\label{tracetau}
\begin{aligned}
\tau (fU^\ast_{\psi}) = \begin{cases}  0,\quad  \text{if}\quad \psi \ne 1,\\
\displaystyle \int_{G_1} f \varpi.
\end{cases}
\end{aligned}
\end{align}
Here $\varpi$ is the volume form attached to the dual basis $\widetilde{\omega}_i$ of $\widetilde{Z}_i$ on $G_1$
\begin{align*}
\varpi =\bigwedge  \widetilde{\omega}_i ,
\end{align*}
and is in fact not always invariant under $G_2$ action, i.e.,
\begin{equation}\label{volumeform1}
\begin{aligned}
&\widetilde{\psi}_{\ast}(\widetilde{Z_i})=U_{\psi}\widetilde{Z_i}U_{\psi}^{\ast }=\Gamma_i^j(\psi^{-1})\widetilde{Z_j} , \\
&\widetilde{\psi}^{\ast}(\widetilde{\omega_i})=(\Gamma_i^j(\psi^{-1})     \circ \widetilde{\psi^{}} )            \widetilde{\omega}_j ,\\
&\widetilde{\psi}^{\ast}(\varpi)=\bigwedge  (\Gamma_i^j(\psi^{-1})     \circ \widetilde{\psi^{}} )  \widetilde{\omega}_j=(\det(\Gamma(\psi^{-1}))    \circ \widetilde{\psi^{}}  )\varpi .
\end{aligned}
\end{equation}
Evaluating at identity gives
\begin{equation}
\begin{aligned}
\widetilde{\psi}^{\ast}(\bigwedge  {\omega}_i)=\det(\gamma)(\psi^{-1}) \bigwedge  {\omega}_i .
\end{aligned}
\end{equation}
We see a representative function $\sigma:=\det(\gamma)$ here. It is easy to see that $\sigma$ is a group-like element in $\cR(G_2)$, i.e., $\sigma(\psi_1\psi_2)=\sigma(\psi_1)\sigma(\psi_2)$. \\
we also have the module function on $\Fg_1$, which is the trace of the adjoint representation $\Fg_1$ on itself, 
\begin{align*}
\delta(Z)=\textrm{Tr} ( ad_Z),
\end{align*}
and we can extend it to a character on $\cU(\Fg_1)$ and then $\cH$ by
\begin{align*}
\delta(F \acl Z_I)= \varepsilon (F) \acl \delta(Z_I),
\end{align*}
so we have a group-like element $\sigma \acl 1$ in $\cH$ and a character $\delta$ on $\cH$.  \\
Let us recall the definition of modular pairs in involution as follows:

\begin{Definition}
Let $\cH$ be a Hopf algebra, $\delta : \cH \to \Cb $ be a character and $\sigma \in \cH$ be a group-like element. The pair $(\delta, \sigma)$ is called a modular pair in involution (MPI) if
\begin{align} \label{mpi_condition}
\delta(\sigma)=1,\quad  S_{\delta}^{2}(h)=\sigma h \sigma ^{-1},
\end{align}
for all $h \in \cH$, where $S_{\delta}$ is the convolution $\delta \Conv S$:
\begin{align*}
S_{\delta}(h)=\delta(h_{(1)})S(h_{(2)}).
\end{align*}
\end{Definition}
A modular pair in involution is the easiest coefficient that we can put for Hopf cyclic cohomology. The generalized coefficients are called stable anti-Yetter-Drinfeld (SAYD) module over $\cH$.  

Assume $\cH$ is a Hopf algebra with invertible antipode, we can define Yetter-Drinfeld module over $\cH$ as in \cite[p.~181]{caenepeel_frobenius_2002}. It is shown in \cite{hajac_stable_2004} that if we replace $S$ with $S^{-1}$, $S^{-1}$ with $S$ in the definition of Yetter-Drinfeld module we can have compatible action and coaction. The new definition gives the so called anti-Yetter-Drinfeld module over $\cH$.

\begin{Definition}
Let $\cH$ be a Hopf algebra with invertible antipode. A right $\cH-$module, left $\cH-$comodule $M$ is called a right-left anti-Yetter-Drinfeld module over $\cH$ if
\begin{align} \label{sayd1}
\Db ( m \tlt h ) =S(h_{(3)}) m_{<-1>} h_{(1)} \otimes m_{<0>} \tlt h_{(2)},\qquad \forall\, m\in M, \, h \in \cH,
\end{align}
moreover, $M$ is called stable if
\begin{align} 
m_{<0>} \tlt m_{<-1>} = m. 
\end{align}
\end{Definition}

\begin{Lemma} \label{lemma2.2}
\cite[Lemma 2.2]{hajac_stable_2004}
Let the ground field $\bC$ be a right module over $\cH$ via a character $\delta$ and a left comodule over H via a group-like $\sigma$. Then $^\sigma \bC _ \delta$ is a stable right-left anti-Yetter-Drinfeld module if and only if $(\delta,\sigma)$ is a modular pair in involution.
\end{Lemma}
\begin{proof}
It's a Lemma in~\cite{hajac_stable_2004}. However we would like to give an illustration here:\\
The second (stability) condition is obvious, so we will just show the first one. On one side, if 
\begin{align*}
S_{\delta}^{2}(h)=\sigma h \sigma ^{-1},
\end{align*}
we calculate 
\begin{align}
\begin{aligned}
S(h_{(3)})1_{<-1>}h_{(1)} \otimes 1_{<0>}\cdot h_{(2)}&=S(h_{(3)})\sigma h_{(1)} \otimes \delta( h_{(2)})\\
&=S_{\delta}(h_{(2)})\sigma h_{(1)} \sigma^{-1} \sigma \otimes 1\\
&=S_{\delta}(h_{(2)}) S_{\delta}^2( h_{(1)} ) \sigma \otimes 1\\
&=S_{\delta}( S_{\delta} (h_{(1)})h_{(2)})  \sigma \otimes 1\\
&=S_{\delta}( \delta (h_{(1)}) S(h_{(2)}) h_{(3)}) \sigma \otimes 1\\
&=S_{\delta}( \delta (h)\eta) \sigma \otimes 1\\
&= \sigma \otimes \delta(h)\\
&=\blacktriangledown(1\cdot h).
\end{aligned}
\end{align}
On the other side, if 
\begin{align*}
\sigma \otimes \delta(h)=S(h_{(3)})\sigma h_{(1)} \otimes \delta( h_{(2)}), 
\end{align*}
then
\begin{align*}
\delta(h)=\sigma^{-1}S_{\delta}(h_{(2)})\sigma h_{(1)},
\end{align*}
we calculate 
\begin{align}
\begin{aligned}
h&=\delta(h_{(1)} S(h_{(2)}))h_{(3)}\\
&=\delta(h_{(1)})\sigma^{-1}S_{\delta}((S(h_{(2)}))_{(2)})\sigma (S(h_{(2)}))_{(1)} h_{(3)} \\
&=\delta(h_{(1)})\sigma^{-1}S_{\delta}(S(h_{(2)(1)}))\sigma (S(h_{(2)(2)}))h_{(3)} \\
&=\sigma^{-1}S_{\delta}( \delta(h_{(1)}) S(h_{(2)}))\sigma \varepsilon(  h_{(3)}) \eta \\
&=\sigma^{-1}S_{\delta}^2 (h)\sigma .
\end{aligned}
\end{align}
\end{proof}
We also have
\begin{Lemma}\label{sayd}
If $^\sigma \bC _ \delta$ is a right-left stable anti-Yetter-Drinfeld module over $\cH$, then $^{S^{-1}(\sigma)} \bC _ \delta$ is a right-left stable anti-Yetter-Drinfeld module over $\cH^{\cop}$.
\end{Lemma}
\begin{proof}
straightforward calculation. Apply $S^{-1}$ to equation~\ref{sayd1}.
\end{proof}
Now we are ready to show that the pair we constructed is a modular pair in involution.
\begin{Proposition}
The pair $(\delta,\sigma^{})$ is a modular pair in involution for the Hopf algebra $\cH$.  
\end{Proposition}
\begin{proof}
\cite[Thm~3.2]{rangipour_van_2012} shows that $(\delta,S\big(\det(\gamma) \big)=\sigma^{-1})$ is a modular pair in involution for the Hopf algebra $\cR(G_2) \acl \cU(\Fg_1)$(we have used antipode in our right coaction so their $\sigma$ is our $\sigma^{-1})$). By~\ref{lemma2.2}, $^{\sigma^{-1}}\hspace{-3pt} \bC _{\delta} $ is a SAYD over $\cR(G_2) \acl \cU(\Fg_1)$. Use previous lemma, $^{\sigma} \bC _{\delta}$ is a right-left SAYD over $\cR(G_2) \acl \cU(\Fg_1)^{\cop}$ and $(\delta,\sigma^{})$ is a modular pair in involution for the Hopf algebra $\cH$.
\end{proof}
Let us now calculate $S_{\delta}$ on these generators for future use.
\begin{align}
\begin{aligned}
S_{\delta}(1 \acl Z_i)&=\mu (\delta \otimes S)\Delta(1 \acl Z_i)\\
&=\mu (\delta \otimes S_{\aclsub}^{-1})\Delta_{\aclsub}^{\rm{op}}(1 \acl Z_i) \\
&=\mu (\delta \otimes S_{\aclsub}^{-1})(1 \acl Z_i \otimes 1 \acl 1 + S(\gamma_i^j) \acl 1 \otimes 1 \acl Z_j)\\
&=\delta(Z_i)   1 \acl 1  +\delta_i^j S_{\aclsub}^{-1} (1 \acl Z_j)\\
&=\delta(Z_i)   1 \acl 1  + S_{\aclsub}^{-1} (1 \acl Z_i)\\
&=\delta(Z_i)   1 \acl 1  + \gamma_i^j\acl S(Z_j), 
\end{aligned}
\end{align}
and
\begin{align}
\begin{aligned}
S_{\delta}(F \acl 1)&=\mu (\delta \otimes S)\Delta(F \acl 1)\\
&=\mu (\delta \otimes S_{\aclsub}^{-1})\Delta_{\aclsub}^{\rm{op}}(F \acl 1) \\
&=\mu (\delta \otimes S_{\aclsub}^{-1})(F_{(2)}\acl 1)\otimes (F_{(1)}\acl 1)\\
&=\varepsilon (F_{(2)}) S(F_{(1)}) \acl 1\\
&=S(F_{}) \acl 1.
\end{aligned}
\end{align}

With a module pair in involution in hand we can associate a cyclic structure to $\cH$ as suggested in~\cite{connes_cyclic_2000}. Specifically, we set $C^n(\cH)=\cH^{\otimes n}$ for $n \ge 1$, and $C^0(\cH)=\bC$. The face maps $\delta_i:C^{n-1}(\cH) \to C^{n}(\cH)$, $0 \le i \le n$, are:
\begin{align}
\begin{aligned}\label{face}
\delta_{0}(h^1 \otimes \cdots \otimes h^{n-1} )&=1\otimes h^1 \otimes \cdots \otimes h^{n-1}, \\
\delta_{i}(h^1 \otimes \cdots \otimes h^{n-1} )&= h^1 \otimes \cdots\otimes \Delta (h^i) \otimes \cdots \otimes h^{n-1}, \quad 1 \le i \le n-1,\\
\delta_{n}(h^1 \otimes \cdots \otimes h^{n-1} )&=h^1 \otimes \cdots \otimes h^{n-1} \otimes \sigma^{} ,
\end{aligned}
\end{align}
for $n  > 1$, and 
\begin{align*}
\delta_{0}(1)&=1 ,\qquad \delta_{1}(1)=\sigma^{},
\end{align*}
for $n=1$.

The degeneracy maps $\sigma_i:C^{n+1}(\cH) \to C^{n}(\cH)$, $0 \le i \le n$, are:
\begin{align}\label{degeneracy}
\sigma_{i}(h^1 \otimes \cdots \otimes h^{n+1} )= h^1 \otimes \cdots\otimes \varepsilon (h^{i+1}) \otimes \cdots \otimes h^{n+1},
\end{align}
for $n  > 0$, and 
\begin{align*}
\sigma_{0}(h)=\varepsilon(h),
\end{align*}
for $n=0$.

The cyclic operator $\tau_{n}:C^{n}(\cH) \to C^{n}(\cH)$, is defined as:
\begin{align}
\begin{aligned}
\tau_{n}(h^1 \otimes \cdots \otimes h^{n} )=&\big( \Delta^{n-1} S_{\delta}(h^1) \big) \cdot h^2 \otimes \cdots\otimes  h^{n} \otimes \sigma^{}, \label{cyclic}\\
=&\sum S(h^1_{(n)})h^2\otimes \cdots \otimes  S(h^1_{(2)})h^n \otimes S_{\delta}(h^1_{(1)})\sigma^{}.
\end{aligned}
\end{align}

The periodic Hopf cyclic cohomology $HP^{\bullet}(\cH;^\sigma \bC _ \delta)$ of $\cH$ with coefficients in the modular pair $(\delta,\sigma)$ is, by definition (cf. \cite{connes_hopf_1998,connes_cyclic_1999}), the $\Zb_2$-graded cohomology of the total complex $CC^{\bullet}(\cH;^\sigma \bC _ \delta)$ associated to the bicomplex $\{CC^{\bullet}(\cH;^\sigma \bC _ \delta), b, B\}$, where

\begin{align}
\begin{aligned}
&CC^{p,q}(\cH;^\sigma \bC _ \delta)=\begin{cases}  0,\quad  \text{if}\quad p < q,\\
\displaystyle C^{p-q}(\cH;^\sigma \bC _ \delta),\quad  \text{if}\quad q \ge p,
\end{cases} \\
&b: C^{n}(\cH;^\sigma \bC _ \delta) \to C^{n+1}(\cH;^\sigma \bC _ \delta)\\
&b=\sum_{i=0}^{n} (-1)^{i}\delta_{i}, \\
&B: C^{n}(\cH;^\sigma \bC _ \delta) \to C^{n-1}(\cH;^\sigma \bC _ \delta)\\
&B= \big(\sum_{i=0}^{n-1}  (-1)^{(n-1)i} \tau^{i}_{n-1} \big) \sigma_{n-1}\tau_{n}(1-(-1)^{n}\tau_{n}) .
\end{aligned}
\end{align}

\begin{Proposition}[\cite{connes_hopf_1998,connes_cyclic_1999}]
Let $\cH$ be a Hopf algebra endowed with a modular pair in involution $(\delta,\sigma^{})$. Then $\cH^{\natural}_{(\delta,\sigma^{})}=\{C^n(\cH)\}_{n \ge 0}$ equipped with the operators given by~\ref{face} to~\ref{cyclic} is a module over the cyclic category $\Lambda$.
\end{Proposition}
\begin{proof} 
See~\cite[Thm~1]{connes_cyclic_2000} for proof. Note it is the MPI condition~\ref{mpi_condition} that ensured the cyclicity.
\end{proof}

Now we will take a deeper look at the Hopf action and the trace we just defined.

\begin{Proposition}
For any $a,b \in \cA$ and $h \in \cH$ one has
\begin{align}
\begin{aligned}
&\tau(ab)=\tau(b \sigma^{}(a)),\\
&\tau(h(a))=\delta(h)\tau(a),\\
&\tau(h(a)b)=\tau(a S_{\delta}(h)(b)),
\end{aligned}
\end{align}
hence $\tau$ is a $\delta$-invariant $\sigma^{}$-trace under the Hopf action.
\end{Proposition}
\begin{proof}
It suffices to verify the identities for non trivial cases. i.e., we assume $a=fU_{\psi}^{\ast}$ and $b=gU_{\psi}$ (recall from our convention \ref{algebramultiplication}, $U_{\psi}=U_{\psi^{-1}}^{\ast}$) for the first and third identities and $a=fU_{1}^{\ast}$ for the second one.\\
\begin{align}
\begin{aligned}
\tau(fU_{\psi}^{\ast} \ast gU_{\psi})&=\int_{G_1} f (g\circ \widetilde{\psi}) \ \varpi\\
&=\int_{G_1} (f\circ \widetilde{\psi^{-1}}) g \ \widetilde{\psi^{-1}}^{\ast} (\varpi)\\
&=\int_{G_1} (f\circ \widetilde{\psi^{-1}}) g \ (\det(\Gamma)(\psi^{})\circ \widetilde{\psi^{-1}}) \varpi\\
&=\int_{G_1} (f\circ \widetilde{\psi^{-1}})(\vp) g (\vp)\ \det(\Gamma)(\psi^{})(  \psi^{-1} \trt       \vp) \varpi (\vp)\\
&=\int_{G_1} (f\circ \widetilde{\psi^{-1}})(\vp) g (\vp)\ \det(\gamma)(\psi^{} \tlt (  \psi^{-1} \trt       \vp)) \varpi (\vp)\\
&=\tau(gU_{\psi} \ast \widetilde {\det (\gamma)^{}} (fU_{\psi}^{\ast})).
\end{aligned}
\end{align}
For the second identity, because of obvious multiplicativity, we verify on generators $1 \acl Z_i $:
\begin{align}
\begin{aligned}
\int_{G_1} \widetilde{Z_i}(f)\ \varpi&= \int_{G_1} \dt f( \vp \exp(tZ_i))\ \varpi(\vp) \\
&=\dt \int_{G_1}  f( \vp \exp(tZ_i))\ \varpi (\vp) \\
&=\dt \int_{G_1}  f( \vp)\ \varpi (\vp \exp(-tZ_i)) \\
&=\dt \Delta(tZ_i) \int_{G_1}  f( \vp)\ \varpi (\vp) \\
&=\delta(Z_i) \int_{G_1} f\ \varpi,
\end{aligned} 
\end{align}
and on $F\acl 1 $:
\begin{align}
\begin{aligned}
\int_{G_1} F((1 \tlt \vp )^{}) f (\vp)\ \varpi (\vp)= F(1) \int_{G_1}  f\ \varpi.
\end{aligned}
\end{align}
For the third identity, because of anti-multiplicativity of $S_{\delta}$, we verify on generators $F \acl 1 $:
\begin{align}
\begin{aligned}
\tau (\widetilde{F \acl 1}(fU_{\psi}^{\ast}) \Conv    gU_{\psi}^{} )
&=\int_{G_1} F((\psi \tlt \vp )^{}) f (\vp) g (\psi \trt \vp) \ \varpi (\vp) \\
&=\int_{G_1} F((\psi^{-1} \tlt (\psi \trt \vp ) )^{-1}) f (\vp) g (\psi \trt \vp) \ \varpi (\vp) \\
&=\int_{G_1} f (\vp) g (\psi \trt \vp)F(( \psi^{-1} \tlt (\psi \trt \vp ) )^{-1})  \ \varpi (\vp) \\
&=\int_{G_1} f (\vp) g (\psi \trt \vp) S(F)(( \psi^{-1} \tlt (\psi \trt \vp) )^{}) \ \varpi (\vp) \\
&=\tau \big(fU_{\psi}^{\ast} \Conv   \widetilde{S_{\delta}( F \acl 1)}( gU_{\psi}^{} ) \big).
\end{aligned}
\end{align}
and on $1\acl Z_i$:
\begin{align}
\begin{aligned}
 \tau ( \widetilde{ 1 \acl Z_i  } (fU_{\psi}^{\ast} \Conv  gU_{\psi}^{}))
=&\int_{G_1} \widetilde{Z_i}(f)(g\circ \widetilde{\psi})\ \varpi\\
=& \int_{G_1} \dt f( \vp \exp(tZ_i)) g (\psi \trt \vp)\ \varpi(\vp) \\
=&\dt \int_{G_1}  f( \vp) g (\psi \trt (\vp \exp(-tZ_i)))\ \varpi(\vp \exp(-tZ_i)) \\
=&\dt \int_{G_1}  f( \vp) g (\psi \trt (\vp \exp(-tZ_i)))\ \varpi(\vp ) \\
&+\dt \int_{G_1}  f( \vp) g (\psi \trt \vp )\ \varpi(\vp \exp(-tZ_i)) \\
=& \int_{G_1}  f( \vp)\ (\Gamma_{i}^j(\psi^{-1} )\circ \widetilde{\psi})(\vp)(-\widetilde{Z}_j(g)\circ \widetilde{\psi}) ( \vp )\ \varpi(\vp ) \\
&+\dt \Delta(tZ_i) \int_{G_1}  f( \vp) g (\psi \trt \vp )\ \varpi(\vp) \\
=& \int_{G_1}  f( \vp)\ \gamma_{i}^j(\psi^{-1} \tlt (\psi \trt \vp ) )(-\widetilde{Z}_j(g)\circ \widetilde{\psi}) ( \vp )\ \varpi(\vp ) \\
&+\delta(Z_i)\int_{G_1}  f( \vp) g (\psi \trt \vp )\ \varpi(\vp) \\
=& \tau ( fU_{\psi}^{\ast} \Conv  ( ( \widetilde{   1 \acl   \gamma  _{i}^j } )(\widetilde{ S(Z_j) \acl 1  })(gU_{\psi}^{})))\\
&+\delta(Z_i)\tau(fU_{\psi}^{\ast} \Conv gU_{\psi})\\
=& \tau ( fU_{\psi}^{\ast} \Conv  ( ( \widetilde{   1 \acl   \gamma  _{i}^j } )(\widetilde{ S(Z_j) \acl 1  })(gU_{\psi}^{})))\\
&+\tau(fU_{\psi}^{\ast} \Conv  \widetilde {\delta(Z_i)}(gU_{\psi}))\\
=& \tau (fU_{\psi}^{\ast} \Conv    \widetilde{      S_{\delta}( Z_i  \acl 1)}(gU_{\psi}^{})),
\end{aligned} 
\end{align}
\end{proof}
\begin{Proposition}[\cite{connes_hopf_1998,connes_cyclic_1999}]
Let $\tau: \cA \to \bC$ be a $\delta$-invariant $\sigma^{}$-trace under the Hopf action of $\cH$ on $\cA$. Then the assignment
\begin{align}
\lambda(h^1 \otimes \cdots \otimes h^n)(a^0,\dots,a^n)=\tau(a^0 h^1(a^1)\cdots h^n(a^n) )
\end{align}
defines a map of $\Lambda$-modules $\lambda^{\natural}: \cH^{\natural}_{(\delta,\sigma^{})} \to \cA^{\natural}$, which induces characteristic homomorphisms in cyclic cohomology:
\begin{align*}
\lambda^{\ast}_{\tau}:HC^{\ast}_{(\delta,\sigma^{})} (\cH) \to HC^{\ast} (\cA).
\end{align*}
\end{Proposition}

\begin{Corollary} 
For any discrete subgroup $\Gamma < G_2^{\delta}$, $\tau: \cA_{\Gamma} \to \bC$ is also a $\delta$-invariant $\sigma^{}$-trace under the Hopf action of $\cH$ on $\cA_{\Gamma}$. The assignment above also defines a map of $\Lambda$-modules $\lambda^{\natural}: \cH^{\natural}_{(\delta,\sigma^{})} \to \cA_{\Gamma}^{\natural}$, which induces characteristic homomorphisms in cyclic cohomology:
\begin{align*}
\lambda^{\ast}_{\tau}:HC^{\ast}_{(\delta,\sigma^{})} (\cH) \to HC^{\ast} (\cA) \to HC^{\ast} (\cA_{\Gamma}).
\end{align*}

\end{Corollary}

\begin{Theorem}\label{injectivehomomorphism}
Let ($G_1$, $G_2$) be a matched pair of Lie groups, $\cA$ be the convolution algebra $C_c^{\ify}(G_1)\rtimes G_2^{\delta}$. $\tau: \cA \to \bC$ be a $\delta$-invariant $\sigma$-trace, as defined in this section, under the Hopf action of $\cH=(\cU(\Fg_1)   \acr  \cR(G_2))^{\cop}$ on $\cA$. Then the assignment
\begin{align}
\lambda(h^1 \otimes \cdots \otimes h^n)(a^0,\dots,a^n)=\tau(a^0 h^1(a^1)\cdots h^n(a^n) ) ,\qquad h^i \in \cH, \quad a^i \in \cA,
\end{align}
defines an injective homomorphism of $\Lambda$-modules $\lambda^{\natural}: \cH^{\natural}_{(\delta,\sigma^{})} \to \cA^{\natural}$.
\end{Theorem}

Before we go to the proof we list two useful lemmas.
\begin{Lemma}
\label{faithfultrace}
the $\sigma^{}-$trace $\tau$ is faithful, i.e., that 
\begin{align*}
\tau (ab)=0 ,\quad \forall a \in \cA, \quad \text{implies}\ b=0.
\end{align*}

\end{Lemma}
\begin{proof}
We just need to show that if $\int_{G_1} f (g\circ \widetilde{\psi}) \ \varpi=0$ for all $g$, then $f=0$, which is the same that if $\int_{G_1} f g \ \varpi=0$ for all $g$, then $f=0$. 

Now if we assume $\int_{G_1} f g \ \varpi=0$ for all $g$, but $f(p)\ne 0$ at some $p \in G_1$. Then there exits some small neighborhood $U$ of $p$ such that $f$ is strictly positive or negative, say $f > c >0$ on $U$. We can have a compact set $V_1\subset U$ and a smaller compact set $V_2\subsetneqq V_1 $. Take $g $ to be a smooth cutoff function that is identity on $V_2$ and vanish off $V_1$, then $\int_{G_1} f g \ \varpi  =\int_{V_1} f g \ \varpi > \int_{V_2} f g \ \varpi > \int_{V_2} c  \ \varpi >0 $, hence a contradiction.     
\end{proof}

\begin{Lemma}\label{n=1}
\begin{align*}
h(a)(e)=0 ,\quad  \forall  a \in \cA  , \quad \text{implies}\ h=0.
\end{align*}
\end{Lemma}
\begin{proof}
Recall~\ref{PBW} and~\ref{basis2}, since the $Z_I$ ’s form a PBW basis of $\cU(\Fg_1)$, and that every element $\sum_{j}F_{j} \acl u_{j}$ in $\cR(G_2) \acl \cU(\Fg_1)$ can be written uniquely as
\begin{align}
\begin{aligned}
\sum_{j}F_{j} \acl u_{j}=\sum_{I} F_{I} \acl Z_{I},\quad I=(i_1,\dots, i_p),
\end{aligned}
\end{align}
we just need to prove that if
\begin{align*}
\sum \widetilde{F_{I}} \widetilde{Z_{I}} (fU_{\psi}^{\ast})(e)=0, \qquad \forall fU_{\psi}^{\ast} \in \cA,
\end{align*}
then $F_{I}=0$ for any $I$.

However, for every $\psi \in G_2$ we will have   
\begin{align*}
\sum F_{I}(\psi \tlt e) \widetilde{Z_{I}} (f)(e)=0, \qquad \forall f \in C_c^{\ify}(G_1),
\end{align*}
Because $Z_{I}$ form a PBW basis of $\cU(\Fg_1)$, we order the sum by the order of the PBW basis assume that the coefficient $F_{I_n}(\psi )$ is non-zero for $I_n$ with top degree. There exits a function $g_n \in C_c^{\ify}(G_1)$ defined on a neighborhood of $e$ such that $\widetilde{Z_{I_n}} (g_n)(e)=1$ and $\widetilde{Z_{I_s}} (g_n)(e)=0$ for all lower $s < n$. Therefore we have $F_{I_n}(\psi)=0$, a contradiction. Therefore $F_{I}(\psi)=0$ for any $I$. Since this is true for every $\psi \in G_2$, we then have our desired conclusion. 
\end{proof}

\begin{proof}
Now we are ready to prove theorem~\ref{injectivehomomorphism}.

Assume $\tau \big( \sum_{j} a^0 h_j^1(a^1)\cdots h_j^n(a^n) \big)=0$ for any $a^0,,\dots,a^n$, by lemma~\ref{faithfultrace}, we conclude
\begin{align}
\sum_{j} h_j^1(a^1)\cdots h_j^n(a^n)=0,
\end{align}
when evaluating at the identity, the above equation becomes 
\begin{align}
\sum_{j} h_j^1(a^1)\cdots h_j^n(a^n)(e)=0,
\end{align}
from where we want to show 
\begin{align*}
\sum_{j} h_j^1\otimes \cdots \otimes h_j^n=0.
\end{align*}
We use induction: When $n=1$, lemma~\ref{n=1} applies. Now assume the statement is true of $n-1$, i.e., if  
\begin{align*}
\sum_{j} h_j^1(a^1)\cdots h_j^{n-1}(a^{n-1})(e)=0,
\end{align*}
we have 
\begin{align}
\sum_{j} h_j^1\otimes \cdots \otimes h_j^{n-1}=0.
\end{align}

Denote $a^k=f_kU_{\psi_k}^{\ast}$. If we use~\ref{basis2} and the PBW basis of $\cU(\Fg_1)$ and write each $h_j^k$ as $F_{j,I(j,k)}^k \acl Z_{j,I(j,k)}^k $. Since $e$ is fixed by every $\psi \in G_2$, when evaluating at the identity the above equality becomes 
\begin{equation*}
\sum_{j,I}  F_{j,I(j,1)}^1 ((\psi_1 \tlt e)^{}) \widetilde{Z_{j,I(j,1)}^1} (f_1)(e) \cdots F_{j,I(j,n)}^n((\psi_n \tlt e)^{}) \widetilde{Z_{j,I(j,n)}^n} (f_n)(e)=0.
\end{equation*}
We partition the sum by the order of $Z_{j,I(j,n)}^n$ and index with $(1\le s \le N,Q(s))$ and use induction again on $s$. after reordering the above equality becomes
\begin{equation*}
\sum_{I_1 \le  I_s \le  I_N} F_{I_s}^n(\psi_n^{} ) \widetilde{Z_{I_s}^n} (f_n)(e)( \sum_{Q(s)} F_{Q}^1 (\psi_1^{} ) \widetilde{Z_{Q}^1} (f_1)(e) \cdots F_{Q}^{n-1}(\psi_{n-1} ^{}) \widetilde{Z_{Q}^{n-1}} (f_{n-1})(e))=0.
\end{equation*}
Now if we look at the top ${Z_{I_N}^n}$ with nonzero $F_{I_N}$, because of the PBW basis, there exits a function $g_n \in C_c^{\ify}(G_1)$ defined on a neighborhood of $e$ such that $\widetilde{Z_{I_N}^n} (g_n)(e)=1$ and $\widetilde{Z_{I_s}^n} (g_n)(e)=0$ for all lower $s < N$. Since we assume $F_{I_N}$ is nonzero, there also exist some $\psi_n^{+}$ such that $F_{I_N}(\psi_n^{+}) \ne 0$.

Now if we fix $a^n=g_nU_{\psi_n^{+}}^{\ast}$, the above equality becomes 
\begin{equation*}
 \sum_{Q(N)} F_{Q}^1 (\psi_1 ^{}) \widetilde{Z_{Q}^1} (f_1)(e) \cdots F_{Q}^{n-1}(\psi_{n-1} ^{}) \widetilde{Z_{Q}^{n-1}} (f_{n-1})(e)=0 \quad \forall f_kU_{\psi_k}^{\ast}, \quad1\le k \le n-1,
\end{equation*}
this is the induction hypothesis for $n-1$ of the first induction, hence we have 
\begin{align}
\sum_{Q(N)} h_Q^1\otimes \cdots \otimes h_Q^{n-1}=0.
\end{align} 
Therefore
\begin{align*}
(\sum_{Q(N)} h_Q^1\otimes \cdots \otimes h_Q^{n-1}) \otimes h_N^n=0.
\end{align*} 
This means the top degree part, with the new order $(1\le s \le N,Q(s))$, of the original sum
\begin{align*}
\sum_{j} h_j^1\otimes \cdots \otimes h_j^n
\end{align*}
is $0$. By the reversed induction on the degree, the whole term is $0$ and $\lambda$ is injective.
\end{proof}

\begin{Definition}\label{representativecochain}
An n-cochain $\vp $ on the algebra $C_c^{\ify}(G_1)\rtimes G_2^{\delta}$ is representative iff it is in the range of the above monomorphism $\lambda$. We denote the representative cochain space as $\text{Im}(\lambda^{\natural})$.
\end{Definition}

\begin{Corollary} \label{representativecochain2}
Under the assumption of the above theorem, $\lambda$ is an isomorphism of $\Lambda$-modules $\lambda^{\natural}: \cH^{\natural}_{(\delta,\sigma^{})} \to \text{Im}(\lambda^{\natural})$.

\end{Corollary}

From now on we do not distinct between $\cH^{\natural}_{(\delta,\sigma^{})}$ and $\text{Im}(\lambda^{\natural})$.

\section{Differentiable and representative cohomologies of action groupoids}
\label{sec:groupoid}

In this section, we will first review Haefliger's differentiable cohomology of a action groupoid~\cite{haefliger_differential_1972}, in particular we take our $G_1 \rtimes G_2^{\delta}$ as example. Then we refine the complex that calculate the differentiable cohomology and define representative cohomology of a action groupoid. The differentiable and representative cohomology works in general for any action groupoid. However, for our example $G_1 \rtimes G_2^{\delta}$ with a matched pair of Lie groups $(G_1,G_2)$, There is an even smaller subcomplex which we called $D^{p,q}$ that satisfies some strong covariance property. We will explain in detail the description of $D^{p,q}$ and its relation (as appeared in the last column of diagram~\ref{diagram2}) between the representative cochain space $\text{Im}(\lambda^{\natural})$ from previous section via Conne's $\Phi$ map.

\subsection{Decomposition of Lie groups}
Recall from~\cite{hochschild_representations_1957} that a {\it nucleus} of a Lie group $G$ is a simply connected solvable closed normal subgroup $L$ of $G$ such that $G/L$ is reductive, in the sense that $G/L$ has a faithful representation and every finite dimensional analytic representation of $G/L$ is semisimple. In this case one proves (\cite[Thm.~9.1]{hochschild_representations_1957}) that $G = H \ltimes L$, where $H\cong G/L$ is the complementary closed subgroup which is reductive. We will assume all groups are connected through this paper. $H$ acts on $L$ from the right by 
\begin{align}
\begin{aligned}
l \tlt h= h^{-1} l h
\end{aligned}
\end{align}
which coincide with the right multiplication of $H$ on $H \backslash G$
\begin{align}
\begin{aligned}
(Hl )h=H(h h^{-1} l h)=H (l \tlt h)
\end{aligned}
\end{align}

Let, in addition, $\Fh,\Fl  \subset \Fg$ be Lie algebras of $H,L$ and $G$ respectively.

Parallel decomposition of Linear algebraic group is called the Levi decomposition (\cite{hochschild_basic_1981} ) where $G = H \ltimes L $, $H$ the levi factor and $L=G_{u}$ the unipotent radical.\\

For a matched pair of Lie groups $G_1$ and $G_2$ we can form the bicrossed product group $G_1\bowtie G_2$ with group structure given by \label{assemble}
\begin{align*}
(\vp_1\bowtie\psi_1)(\vp_2\bowtie\psi_2)=(\vp_1 \psi_1 \trt \vp_2)\bowtie(\psi_1 \tlt \vp_2 \psi_2).
\end{align*}

For a matched pair of Lie algebras $\Fg_1$ and $\Fg_2$ we can also form the bicrossed product Lie algebra $\Fg_1\bowtie \Fg_2$ with underlying vector space $\Fg_1\oplus \Fg_2$ and bracket relation:
\begin{equation*}
[X_1 \oplus Y_1, X_2 \oplus Y_2] = ([X_1,x_2]+ Y_1 \trt X_2 -Y_2 \trt X_1) \oplus ( [Y_1,Y_2]+Y_1 \tlt X_2 -Y_2 \tlt X_1 ).
\end{equation*}

Now let $G_2 = H_2 \ltimes L_2$ be such decomposition of $G_2$. In general, the simply-connected nucleus $L_2$ of a Lie group is differ from nilpotent by a vector group, to simplify the situation we assume further $L_2$ is nilpotent with Lie algebra $\Fl_2$. In the algebraic case $L_2$ as the unipotent radical is nilpotent. Under our assumption, for both case we have $L_2=\exp (\Fl_2)$ and $\Fl_2=\mathfrak{nil}(\Fg_2)$.  \\

\subsection{Differentiable cohomology of action groupoid}

Let $\cG=G_1 \rtimes G_2^{\delta}$ be our action groupoid, $\Omega_{G_1}^{\ast}$ be the complex of differential forms on $G_1$. We introduce the following double complex $\label{bigcomplex}C^{p,q},d_1,d_2$. $C^{p,q}= \{ 0 \} $ unless $p \ge 0$ and $0 \le q \le \text{ dim }G_1 $, $C^{p,q}$ be the space of totally antisymmetric maps $\alpha:G_2^{p+1}\to \Omega^{q}(G_1)$ that satisfy the $G_2$-equivariant condition
\begin{align}
\alpha(\psi_0  \psi,\dots,\psi_p \psi )=(\psi^{-1}\trt)^{\ast} \alpha(\psi_0,\dots,\psi_p), \quad \forall \psi,\psi_i \in G_2.
\end{align}
The coboundary $d_1:C^{p,q} \to C^{p+1,q}$ is given by
\begin{align}
(d_1\alpha)(\psi_0,\dots,\psi_{p+1})=\sum_{j=0}^{p+1}(-1)^{q+j}\alpha(\psi_0,\dots,\check{\psi_{j}} , \dots,\psi_{p+1}),
\end{align}
The coboundary $d_2:C^{p,q} \to C^{n,q+1}$ is given by de Rham boundary:
\begin{align}
(d_2\alpha)(\psi_0,\dots,\psi_{p+1})=d \alpha(\psi_0,\dots,\psi_{p+1}).
\end{align}

The complex $C^{p,q} $ has a subcomplex $\label{diffcomplex}C_{d}^{p,q} $, which consists of forms $\alpha(\psi_0 ,\dots,\psi_p )$ that depend on $\psi_i$ {\it smoothly}, i.e., we can write $\alpha \in C_{d}^{p,q}$ as 

\begin{align}
\alpha(\psi_0,\dots,\psi_p)=\sum c_j \rho_j  ,
\end{align}
here $\rho_j$ form a basis of left $G_1$ invariant forms on $G_1$, and $c_j$ are smooth functions on $G_2$ .

\begin{Definition}\cite{haefliger_differential_1972}
Let $\cG=G_1 \rtimes G_2^{\delta}$. The cohomology of the action groupoid $H^{\bullet} ( \cG;\Rb)$ is the cohomology of the simple complex associated to the above double complex $C^{p,q}$ and the differentiable cohomology $H^{\bullet}_{d}( \cG;\Rb)$ is the cohomology of the simple complex associated to the subcomplex $C_{d}^{p,q}=C^{p}_{d}( G_2^{\delta};\Omega_{G_1}^{q})$ whose elements are forms depending smoothly on $p$ elements of $G_2$ as in~\ref{diffcomplex}.
\end{Definition}

\begin{Proposition}\cite[Thm4.2]{haefliger_differential_1972}
Let $\cG=G_1 \rtimes G_2^{\delta}$. $H^{\bullet}_{d}( \cG;\Rb)$ is isomorphic to the cohomology of $G_2$-invariant forms on $G_2/K_2 \times G_1$, where $K_2$ is a maximal compact subgroup of $G_2$, and $G_2$ acts diagonally.
\end{Proposition}

\subsection{Representative cohomology of action groupoid}
We refine Haefliger's differentiable cohomology of groupoid by defining a subcomplex, $C^{p}_{\cR}( G_2^{\delta};\Omega_{G_1}^{q})$ of $C^{p}_{d}( G_2^{\delta};\Omega_{G_1}^{q})$. The elements of $C_{\cR}^{p,q} =C^{p}_{\cR}( G_2^{\delta};\Omega_{G_1}^{q})$ are forms depending {\it representatively} on $p+1$ elements of $G_2$. i.e., the coefficient functions are products of $p+1$ representative functions on $G_2$, we can write $\alpha \in C_{\cR}^{p,q}$ as 
\begin{align}
\alpha(\psi_0,\dots,\psi_p)=\sum c_j \rho_j  ,
\end{align}
where $c_j$ are representative functions on $G_2$.

Since representative functions are smooth, we can get a subcomplex and define the {\it representative cohomology $H^{\bullet}_{\cR}( \cG;\Rb)$}:

\begin{Definition}
Let $\cG=G_1 \rtimes G_2^{\delta}$. The representative cohomology $H^{\bullet}_{\cR}( \cG;\Rb)$ is the cohomology of the simple complex associated to the double complex $C^{p}_{\cR}( G_2^{\delta} ;\Omega_{G_1}^{q})$.
\end{Definition}

\begin{Theorem}\label{representative cohomology}
Let $\cG=G_1 \rtimes G_2^{\delta}$.  The representative cohomology $H^{\bullet}_{\cR}( \cG;\Rb)$ is isomorphic to the cohomology of $G_2$-invariant forms on $L_2 \times G_1$, where $L_2$ is the nucleus of $G_2$, and $G_2$ acts diagonally.
\end{Theorem}

\begin{proof}

Step 1: \\
Consider the first inclusion
\begin{align*}
\varepsilon_1 : C^{p}_{\cR}( G_2^{\delta};\Omega_{G_1}^{q}) \to C^{p}_{\cR}( G_2^{\delta} ;\Omega_{(L_2\times G_1)}^{q})
\end{align*}
induced by the projection $\pi: L_2 \times G_1 \to G_1$ induces an isomorphism of the cohomology of the double complexes.

By a familiar spectral sequence argument (\cite[p.~86]{godement_topologie_1958}), we only need to show that the inclusion above induces an isomorphism of the vertical deRham cohomology. This is true because the fiber of $\pi$ is smoothly contractible.\\

Step 2: \\
Follow the usual double complex argument as in~\cite{bott_differential_1995}. The augmented sequence of the horizontal complex is exact for every $q$:
\begin{align*}
0 \to (\Omega_{(L_2\times G_1)}^{q})^{G} \xrightarrow{\varepsilon_2} C^{0}_{\cR}( G_2^{\delta} ;\Omega_{(L_2\times G_1)}^{q}) \xrightarrow{d_1} C^{1}_{\cR}( G_2^{\delta} ;\Omega_{(L_2\times G_1)}^{q}) \xrightarrow{d_1} \cdots
\end{align*} 
where $(\Omega_{(L_2\times G_1)}^{q})^{G_2}$ are $G_2$-invariant forms on $L_2 \times G_1$, $\varepsilon_2$ is another inclusion. Because of this exactness we can use the usual diagram chasing argument for double complex with exact rows to get the fact that $\varepsilon_2$ induces an isomorphism from the cohomology of augmented column complex to the cohomology of double complex $C^{p}_{\cR}( G_2^{\delta} ;\Omega_{(L_2\times G_1)}^{q})$. \\

Step 3: \\
Finally the exactness of the above sequence is provided by the given homotopy (cf.,~\cite[Lemma~2.3]{kumar_extensions_2006}):
\begin{align}
\begin{aligned}
&H:\ C^{p}_{\cR}( G_2^{\delta} ;\Omega_{(L_2\times G_1)}^{q})\to \ C^{p-1}_{\cR}( G_2^{\delta} ;\Omega_{(L_2\times G_1)}^{q})\\
&H(\alpha)(\psi_1,\dots,\psi_{p})(l\times m)=(-1)^q \pi(f_{(\psi_1,\dots,\psi_{p})}^\alpha),
\end{aligned}
\end{align}
where
\begin{align}
\begin{aligned}
f_{(\psi_1,\dots,\psi_{p})}^\alpha (h):=\alpha(hl,\psi_1,\dots,\psi_{p})(l\times m),
\end{aligned}
\end{align}
and $\pi$ the projection to its $H_2$ invariant part, since $H_2$ is reductive and every $H_2$ module is completely reducible.
We have
\begin{align}
\begin{aligned}
&H d_1(\alpha)(\psi_0,\dots,\psi_{p})(l\times m)\\
=&(-1)^q \pi( f_{(\psi_0,\psi_1,\dots,\psi_{p})}^{ d_1 \alpha}) \\
=& \alpha(\psi_0,\dots,\psi_{p})(l\times m) - \pi( f_{(\psi_1,\dots,\psi_{p})}^{  \alpha}) +\sum_{j=1}^{p}(-1)^{j+1}\pi( f_{(\psi_0,\dots \check{\psi_j} \dots,\psi_{p})}^{  \alpha}) ,
\end{aligned}
\end{align}
\begin{align*}
d_1 H (\alpha)(\psi_0,\dots,\psi_{p})(l\times m)=\sum_{j=0}^{p}(-1)^{j}\pi( f_{(\psi_0,\dots \check{\psi_j} \dots,\psi_{p})}^{  \alpha}) .
\end{align*}
Therefore
\begin{align}
H d_1 + d_1 H =I.
\end{align}
\end{proof}

\subsection{A smaller subcomplex \texorpdfstring{$D ^{p,q} $}{Dpd} and Connes' \texorpdfstring{$\Phi$}{Phi} map}

The complex $C_{\cR}^{p,q} $ has a subcomplex $D ^{p,q} $, which consists of forms $\alpha(\psi_0 ,\dots,\psi_p )$ that can be written as 
\begin{align}
\alpha(\psi_0,\dots,\psi_p)=\sum c_j \rho_j  ,
\end{align}

where $c_j$ are smooth functions on $G_1$ which are finite linear combinations of finite products of the following functions,
\begin{align*}
\vp  \in  G_1 \to F_i \big( (\psi_i \tlt \vp)  \big),\quad F_i \in \cR(G_2) ,
\end{align*}
more precisely,
\begin{align}
c_j(\vp)=\sum_{\text{finite}} \prod_{i=0}^{p}F_i \big( (\psi_i \tlt \vp)\big),\quad F_i \in \cR(G_2).
\end{align}
The important fact about $\alpha$ is that it is not only $G_2$-equivariant but also $G$-equivariant, i.e.,
\begin{align}
\alpha(\psi_0 \tlt \phi,\dots,\psi_p \tlt \phi )=(\phi^{-1} \trt)^{\ast}   \alpha(\psi_0,\dots,\psi_p), \quad \forall \psi_i \in G_2,\phi=\vp\psi \in G_1G_2=G.
\end{align}

We use Connes' $\Phi$ map from bicomplex $(C^{p,q},d_1,d_2)$ to the $(b,B)$ bicomplex of the algebra $\cA=C_c^{\ify}(G_1)\rtimes G_2^{\delta}$ and we show that the resulting cochains are representative as in~\ref{representativecochain}, i.e., in the range of $\lambda$. Hence we have $\Phi$ from $(C^{p,q},d_1,d_2)$ to the $(b,B)$ bicomplex of $\cH$.

Let us recall the construction of $\Phi$. As in~\cite{connes_noncommutative_1994}, we let $\mathcal{B}$ be the tensor product,
\begin{align*}
\mathcal{B}=A^\ast(G_1) \otimes \Lambda^{\ast} (\bC (G_2')),
\end{align*}
where $A^\ast(G_1)$ is the algebra of smooth forms with compact support on $G_1$, the exterior algebra $\Lambda^{\ast} (\bC (G_2'))$ of the linear space $\bC (G_2')$ with basis the $\delta_{\psi}, \psi \in G_2$, with $\delta_e=0$. We take the crossed product  
\begin{align*}
\mathcal{C}=\mathcal{B} \rtimes  G_2.
\end{align*}
The action of $G_2$ on $\mathcal{B}$ by automorphisms is defined as 
\begin{align}
\begin{aligned}
U^{\ast}_{\psi} \omega U^{}_{\psi} = \widetilde{\psi}^{\ast} \omega = \omega \circ \widetilde{\psi} \qquad \forall \omega \in A^\ast(G_1) , \\
U^{\ast}_{\psi} \delta_{\psi_2} U^{}_{\psi} =\delta_{\psi_2 \psi} -\delta_{\psi}  \qquad \forall \psi_2 \in G_2.
\end{aligned}
\end{align}
We write the elements of $\mathcal{C}$ as a finite sum
\begin{align*}
c=\sum_{\psi} b U^{\ast}_{\psi}, \qquad b \in \mathcal{B},
\end{align*}
then we can write the differential $d$ in $\mathcal{C}$ as  
\begin{align}
d (b U^{\ast}_{\psi})= d b U^{\ast}_{\psi}  - (-1)^{\partial b}b \delta_{\psi} U^{\ast}_{\psi},
\end{align}
where the first term is only the exterior differential in $A^\ast(G_1)$.

A cochain $\gamma$ in $C^{p,q}$ determines a form $\widetilde{\gamma}$ on $\mathcal{C}$ by
\begin{align}
\widetilde{\gamma}(\omega\otimes\delta_{\psi_1}\cdots \delta_{\psi_p} )=\int_{G_1}\omega \wedge \gamma(e,\psi_1,\dots,\psi_p),\quad \widetilde{\gamma}(b U^{\ast}_{\psi})=0 \quad \text{if}\ \psi \ne e.
\end{align}
we define $\Phi(\gamma)$ by
\begin{align}
\begin{aligned}
\Phi ( \gamma )  (a^0,\dots,a^{l}) &=\frac{p!}{(l+1)!} \sum_{j=0}^{l} (-1)^{j(l-j)}\widetilde{\gamma}(da^{j+1}\cdots da^{l} a^0 da^{1}\cdots da^{j}) \\
\text{where}\quad l&=p+\text{dim}(G_1)-q,\qquad a^0,\dots,a^{l} \in \cA.
\end{aligned}
\end{align}
\newline
We have the inclusions:
\begin{align}
\begin{aligned}
D^{p,q} \subset C^{p,q}_{\cR} \subset C^{p,q}_{d} \subset C^{p,q}
\end{aligned}
\end{align}
and we use $\Phi_{d}$, $\Phi_{\cR}$, $\Phi_{D}$ to denote the restriction of $\Phi$ on these above subcomplexes, they are still chain maps because $(b,B)$ maps will not change the type of the cochains.

\begin{Lemma}
Images of $\Phi_{D}$ are representative cochains as defined in~\ref{representativecochain}.
\end{Lemma}
\begin{proof}
we follow exactly the same argument as in~\cite[p.~233-234]{connes_hopf_1998}. When we need to rearrange the terms we use the coproduct of $\cR(G_2)$, when we need to permute functions and $U^{\ast}_{\psi}$ we use the fact that the left and right translates of representative function is still a representative function. 
\end{proof}

\begin{Proposition}
$\Phi_{D}$ is a chain map from the bicomplex $\big( D^{p,q} ,d_1,d_2 \big)$ to the $(b,B)-$complex of the Hopf algebra $\cH$.
\end{Proposition}
\begin{proof}

It is proved in~\cite[III.2.$\delta$,Thm.~14]{connes_noncommutative_1994} that $\Phi$ is a chain map from the bicomplex $\big( C^{p,q} ,d_1,d_2 \big)$ to the $(b,B)-$complex of the algebra $\cA$. From the lemma above, the restriction $\Phi_{D}$ maps $D^{p,q}$ to representative cochain space, i.e., $Im(\lambda^{\natural})$, which is identified with $\cH^{\natural}_{(\delta,\sigma^{})}$ as in Corollary~\ref{representativecochain2}. 

\end{proof}

\section{Explicit construction of Hopf cyclic classes}
\label{sec:cyclic}

As in diagram~\ref{diagram1}, the work of Rangipour and S\"{u}tl\"{u} (\cite{rangipour_van_2012}) can be split into two steps: a vertical map from the Hopf cyclic complex of $\cR(G_2)  \acl \cU(\Fg_1) $ to the complex $C^{p}_{\cR}(G_2,\wedge ^{q} \Fg_1^{\ast})$ and a horizontal map $\nu$ from $C^{p}_{\cR}(G_2,\wedge ^{q} \Fg_1^{\ast})$ to the relative Lie algebra complex. The first process involves a number of quasi-isomorphic complexes and maps, hence will be recalled at the beginning of this section. We then construct an isomorphism $\Theta$ that link $C^{p}_{\cR}(G_2,\wedge ^{q} \Fg_1^{\ast})$ with $D^{p,q}$ from previous section. In the middle of the section we construct an one-sided inverse of $\nu$ map, which we will call $\mathcal{E}$. To the end we prove the main theorem of our paper: $ \mathcal{E} \circ \Theta  \circ \mathcal{D} $ is a quasi-isomorphism, which goes along the opposite direction comparing to the maps of Rangipour and S\"{u}tl\"{u}, that can be used to transit cohomology classes from Lie algebra cohomology complex to representative cocycles on the convolution algebra explicitly.

\subsection{Bicomplexes that calculate Hopf cyclic cohomology of \texorpdfstring{$\cH$}{H}}
\label{subsec:bicomplexes}
In this section we will review a few quasi-isomorphic bicomplex that are all quasi-isomorphic to the cyclic complex of $\cH$. We refer to~\cite{moscovici_hopf_2009,moscovici_hopf_2011,rangipour_van_2012} for the detailed discussion. The first step is taken from~\cite{rangipour_van_2012}, but with our setting; while the last three steps are extra steps from~\cite{moscovici_hopf_2009,moscovici_hopf_2011}, also with our setting, that help us to link to $D^{p,q}$.\\
\newline
step 1
\newline
From $CC^{\bullet}(\cH,^{\sigma^{}} \bC _ \delta)$ to $CC^{\bullet}(\cH^{\cop},^{\sigma^{-1}}\hspace{-3pt} \bC _ \delta)=CC^{\bullet}(\cR(G_2) \acl \cU(\Fg_1),^{\sigma^{-1}}\hspace{-3pt} \bC _ \delta)$.
\newline
It was shown in \cite[Prop3.1]{moscovici_hopf_2009} that 
\begin{align*}
&\mathcal{T}:CC^{\bullet}(\cH,^{\sigma^{}} \bC _ \delta) \to \ CC^{\bullet}(\cH^{\cop},^{\sigma^{-1}}\hspace{-3pt} \bC _ \delta),\\
&\mathcal{T}(1\otimes h^0 \otimes h^1 \times \cdots \otimes h^n)=1\otimes {\sigma^{-1}} h^0 \otimes h^n \otimes \cdots \otimes h^1
\end{align*}
defines an isomorphism of mixed complexes.\\

Recall that we have used antipode in our right coaction so the element $\sigma$ in~\cite{rangipour_van_2012} is our $\sigma^{-1}$. It is proved in~\cite[4.2]{rangipour_van_2012} that our $CC^{\bullet}(\cR(G_2) \acl \cU(\Fg_1),^{\sigma^{-1}}\hspace{-3pt} \bC _ \delta)$ is quasi-isomorphic to a double complex \label{doublecomplex2}
\begin{equation}
\begin{tikzpicture}[description/.style={fill=white,inner sep=2pt}]
\matrix (m) [matrix of math nodes, row sep=3em, column sep=3em, 
text height=1.5ex, text depth=0.25ex]
{   \vdots & \vdots & \vdots    \\
     \wedge^{2} \Fg_1^{\ast} & \wedge^{2} \Fg_1^{\ast} \otimes \cR &  \wedge^{2} \Fg_1^{\ast} \otimes \cR^{\otimes 2} &  \cdots   \\ 
     \Fg_1^{\ast} & \Fg_1^{\ast} \otimes \cR &  \Fg_1^{\ast} \otimes \cR^{\otimes 2} &  \cdots   \\ 
    \bC & \bC \otimes \cR &  \bC \otimes \cR^{\otimes 2} &  \cdots   \\  };
  
\path[transform canvas={yshift=0.6ex},->,font=\scriptsize]
(m-2-1) edge node[above] {$ b_{\cR} $} (m-2-2)  (m-2-2) edge node[above] {$ b_{\cR} $} (m-2-3) (m-2-3) edge node[above] {$ b_{\cR} $} (m-2-4)  
(m-3-1) edge node[above] {$ b_{\cR} $} (m-3-2)  (m-3-2) edge node[above] {$ b_{\cR} $} (m-3-3) (m-3-3) edge node[above] {$ b_{\cR} $} (m-3-4)  
(m-4-1) edge  node[above] {$ b_{\cR} $} (m-4-2)  (m-4-2) edge node[above] {$ b_{\cR} $} (m-4-3) (m-4-3) edge node[above] {$ b_{\cR} $} (m-4-4)    ;

\path[transform canvas={xshift=0ex},->,font=\scriptsize]
(m-4-1) edge node[left]{$\partial_{\Fg}$} (m-3-1)  (m-4-2) edge node[left]{$\partial_{\Fg}$} (m-3-2)   (m-4-3) edge node[left]{$\partial_{\Fg}$} (m-3-3)  
(m-3-1) edge node[left]{$\partial_{\Fg}$} (m-2-1)  (m-3-2) edge node[left]{$\partial_{\Fg}$} (m-2-2)   (m-3-3) edge node[left]{$\partial_{\Fg}$} (m-2-3) 
(m-2-1) edge node[left]{$\partial_{\Fg}$} (m-1-1)  (m-2-2) edge node[left]{$\partial_{\Fg}$} (m-1-2)   (m-2-3) edge node[left]{$\partial_{\Fg}$} (m-1-3) ;

\path[transform canvas={yshift=-0.6ex},->,font=\scriptsize]
(m-2-2) edge node[below] {$ B_{\cR} $} (m-2-1)  (m-2-3) edge node[below] {$ B_{\cR} $} (m-2-2) (m-2-4) edge node[below] {$ B_{\cR} $} (m-2-3)  
(m-3-2) edge node[below] {$ B_{\cR} $} (m-3-1)  (m-3-3) edge node[below] {$ B_{\cR} $} (m-3-2) (m-3-4) edge node[below] {$ B_{\cR} $} (m-3-3)  
(m-4-2) edge node[below] {$ B_{\cR} $} (m-4-1)  (m-4-3) edge node[below] {$ B_{\cR} $} (m-4-2) (m-4-4) edge node[below] {$ B_{\cR} $} (m-4-3)    ;

\end{tikzpicture} 
\end{equation}

the coboundary $\partial_{\Fg}$ is the Lie algebra cohomology coboundary of $\Fg_1$ with coefficients in $\cR ^{\otimes \ast}$ with right $\Fg_1$ action given by
\begin{align}
&( F^1 \otimes \cdots \otimes F^p) \tlt  Z=- Z \trt (F^1 \otimes \cdots \otimes F^p),
\end{align}
here the left action of $\Fg_1$ on $\cR ^{\otimes p}$ is extended from \ref{leftaction} but is different from the usual diagonal action. Because the coproduct of $\Fg_1$ we used here is not from $\cU(\Fg_1)$ but from $\cR(G_2)   \acl  \cU(\Fg_1)$.
The formula is given by
\begin{align*}
&Z \trt (F^1 \otimes \cdots \otimes F^p)= \\
&Z_{(1)_{<0>}} \trt F^1 \otimes Z_{(1)_{<1>}} (Z_{(2)_{<0>}} \trt F^2) \otimes \cdots \otimes Z_{(1)_{<p-1>}}\cdots Z_{(p-1)_{<1>}}(Z_{(p) }\trt F^p).
\end{align*}
The coboundary $b_{\cR}$ involve the coaction $\blacktriangledown$. It is the $b$ operator for coalgebra $\cR$ with coefficients in $\wedge^{\bullet} \Fg_1^{\ast}$. The left coaction is given by~\ref{leftcoaction2}
\begin{align*}
\blacktriangledown (\omega ^i)= S(\gamma^i_j) \otimes \omega^j  ,
\end{align*}
and then extend to $\wedge^{\bullet} \Fg_1^{\ast}$.

It has the explicit expression
\begin{align}
\begin{aligned}
& b_{\cR} ( \alpha \otimes F^1 \otimes \cdots \otimes F^p)=  \alpha \otimes 1 \otimes F^1 \otimes \cdots \otimes F^p \\
&+ \sum\limits_{1\le i \le p}^{} (-1)^{i}  \alpha \otimes F^1 \otimes \cdots \otimes \Delta (F^i) \otimes \cdots \otimes F^p \\
& +(-1)^{p+1} \alpha_{<0>} \otimes F^1 \otimes \cdots \otimes F^p \otimes \alpha_{<-1>}  ;
\end{aligned}
\end{align}
 
\begin{align}
\begin{aligned}
& B_{\cR} =(\sum_{i=0}^{p-1}(-1)^{(p-1)i}\tau_{\cR}^{i})\sigma_{\cR}\tau_{\cR} (1-(-1)^{p} \tau_{\cR}),\qquad \text{with}\\
& \tau_{\cR}( \alpha \otimes F^1 \otimes \cdots \otimes F^p)=\alpha_{<0>} \otimes S(F^1)\cdot (F^2 \otimes \cdots \otimes F^p \otimes \alpha_{<-1>})\\
& \sigma_{\cR}( \alpha \otimes F^1 \otimes \cdots \otimes F^p)=\varepsilon (F^p)\alpha  \otimes F^1 \otimes \cdots \otimes F^{p-1}.
\end{aligned}
\end{align}

Here we see no $\sigma^{-1}$ or $\delta$, because during the transition we have identified the module $^{\sigma^{-1}}\hspace{-3pt} \bC _{\delta} $ with $\wedge^{\text{\tiny dim} (\Fg_1)} \Fg_1^{\ast}$ via Poincar\'{e} isomorphism. \\

In~\cite{rangipour_van_2012}, the authors stopped at this complex and derived the desired quasi-isomorphism between relative Lie algebra cohomology and the Hopf cyclic cohomology that we begin from. However, we want to go to $C^{p}_{\cR}(G_2,\wedge ^{q} \Fg_1^{\ast})$ and $D^{p,q}$. Therefore we pass from the previous complex to its homogeneous version and then to the equivariant cochains. \\
\newline
step 2
\newline
homogeneous version\\

We define $C^{p,q}_{\cR}(\wedge \Fg_1^{\ast},\otimes \cR(G_2)):=(\wedge^{q} \Fg_1^{\ast}\otimes \cR(G_2)^{\otimes p+1})^{\cR(G_2)}$. An element $\sum \alpha \otimes F^0 \otimes \cdots \otimes F^p$ is in $(\wedge^{p} \Fg_1^{\ast}\otimes \cR(G_2)^{\otimes p+1})^{\cR(G_2)}$ if it satisfies the $\cR-$coinvariance condition:
\begin{equation*}
\sum \alpha_{<0>} \otimes F^0 \otimes \dots \otimes F^p \otimes \alpha_{<-1>}   
 = \sum \alpha \otimes F^0_{\ (1)} \otimes F^1_{\ (1)} \otimes \dots \otimes F^p_{\ (1)} \otimes F^0_{\ (2)}\cdots F^p_{\ (2)}.
\end{equation*}
The two complexes are isomorphic via
\begin{align}
\begin{aligned}
&\mathcal{I}: \wedge^{q} \Fg_1^{\ast}\otimes \cR(G_2)^{\otimes p} \to (\wedge^{q} \Fg_1^{\ast}\otimes \cR(G_2)^{\otimes p+1})^{\cR(G_2)}, \\
&\mathcal{I}(\alpha \otimes F^1 \otimes \cdots \otimes F^p)= \\
&\alpha_{<0>} \otimes F^1_{\ (1)} \otimes S(F^1_{\ (2)} ) F^2_{\ (1)} \otimes \dots  \otimes S(F^{p-1}_{\ (2)} ) F^p_{\ (1)} \otimes S(F^p_{\ (2)}) \alpha_{<-1>} ;
\end{aligned}
\end{align}
and 
\begin{align}
\begin{aligned}
&\mathcal{I}^{-1}:  (\wedge^{q} \Fg_1^{\ast}\otimes \cR(G_2)^{\otimes p+1})^{\cR(G_2)} \to \wedge^{q} \Fg_1^{\ast}\otimes \cR(G_2)^{\otimes p}, \\
&\mathcal{I}^{-1}(\alpha \otimes F^0 \otimes \cdots \otimes F^p)= \\
&\alpha  \otimes F^0_{\ (1)} \otimes F^0_{\ (2)}  F^1_{\ (1)} \otimes F^0_{\ (3)}  F^1_{\ (2)} F^2_{\ (1)} \otimes \dots  \otimes F^{0}_{\ (p)} \cdots  F^{p-2}_{\ (2)}  F^{p-1} \varepsilon (F^{p}).
\end{aligned}
\end{align}

\begin{equation}
\begin{tikzpicture}[description/.style={fill=white,inner sep=2pt}]
\matrix (m) [matrix of math nodes, row sep=3em, column sep=3em, 
text height=1.5ex, text depth=0.25ex]
{   \vdots & \vdots & \vdots    \\
     \wedge^{2} \Fg_1^{\ast} & (\wedge^{2} \Fg_1^{\ast}\otimes \cR^{\otimes 2})^{\cR} &  (\wedge^{2} \Fg_1^{\ast}\otimes \cR^{\otimes 3})^{\cR} &  \cdots   \\ 
     \Fg_1^{\ast} & (\wedge^{1} \Fg_1^{\ast}\otimes \cR^{\otimes 2})^{\cR} &  (\wedge^{1} \Fg_1^{\ast}\otimes \cR^{\otimes 3})^{\cR} &  \cdots   \\ 
    \bC & ( \bC \otimes \cR^{\otimes 2})^{\cR} &  ( \bC \otimes \cR^{\otimes 3})^{\cR} &  \cdots   \\  };
  
\path[transform canvas={yshift=0.6ex},->,font=\scriptsize]
(m-2-1) edge node[above] {$ b_{\cR} $} (m-2-2)  (m-2-2) edge node[above] {$ b_{\cR} $} (m-2-3) (m-2-3) edge node[above] {$ b_{\cR} $} (m-2-4)  
(m-3-1) edge node[above] {$ b_{\cR} $} (m-3-2)  (m-3-2) edge node[above] {$ b_{\cR} $} (m-3-3) (m-3-3) edge node[above] {$ b_{\cR} $} (m-3-4)  
(m-4-1) edge  node[above] {$ b_{\cR} $} (m-4-2)  (m-4-2) edge node[above] {$ b_{\cR} $} (m-4-3) (m-4-3) edge node[above] {$ b_{\cR} $} (m-4-4)    ;

\path[transform canvas={xshift=0ex},->,font=\scriptsize]
(m-4-1) edge node[left]{$\partial_{\Fg}$} (m-3-1)  (m-4-2) edge node[left]{$\partial_{\Fg}$} (m-3-2)   (m-4-3) edge node[left]{$\partial_{\Fg}$} (m-3-3)  
(m-3-1) edge node[left]{$\partial_{\Fg}$} (m-2-1)  (m-3-2) edge node[left]{$\partial_{\Fg}$} (m-2-2)   (m-3-3) edge node[left]{$\partial_{\Fg}$} (m-2-3) 
(m-2-1) edge node[left]{$\partial_{\Fg}$} (m-1-1)  (m-2-2) edge node[left]{$\partial_{\Fg}$} (m-1-2)   (m-2-3) edge node[left]{$\partial_{\Fg}$} (m-1-3) ;

\path[transform canvas={yshift=-0.6ex},->,font=\scriptsize]
(m-2-2) edge node[below] {$ B_{\cR} $} (m-2-1)  (m-2-3) edge node[below] {$ B_{\cR} $} (m-2-2) (m-2-4) edge node[below] {$ B_{\cR} $} (m-2-3)  
(m-3-2) edge node[below] {$ B_{\cR} $} (m-3-1)  (m-3-3) edge node[below] {$ B_{\cR} $} (m-3-2) (m-3-4) edge node[below] {$ B_{\cR} $} (m-3-3)  
(m-4-2) edge node[below] {$ B_{\cR} $} (m-4-1)  (m-4-3) edge node[below] {$ B_{\cR} $} (m-4-2) (m-4-4) edge node[below] {$ B_{\cR} $} (m-4-3)    ;

\end{tikzpicture} 
\end{equation}

We transform these boundary operators via $\mathcal{I}$ and get 

the coboundary $\partial_{\Fg_1}$ is the Lie algebra cohomology coboundary of $\Fg_1$ with coefficients in $ \cR ^{\otimes \bullet}$ with right $\Fg_1$ action now given by the usual diagonal action:
\begin{align*}
&(1\otimes F^1 \otimes \cdots \otimes F^p) \tlt  Z\\
&= -1\otimes (\sum_{i=1}^{p}  F^1 \otimes \cdots \otimes Z \trt F^i \otimes \cdots \otimes F^p)
\end{align*}
while $b_{\cR}$ has a simple expression,
\begin{align}
\begin{aligned}
 b_{\cR} ( \alpha \otimes F^0 \otimes \dots \otimes F^p)= \sum_{i=0}^{p+1} (-1)^i \alpha \otimes F^0 \otimes \dots \otimes F^{i-1} \otimes 1 \otimes F^{i} \otimes
\dots \otimes F^{p}.
\end{aligned}
\end{align}
$B_{\cR}$ also has a simple expression,
\begin{align}
\begin{aligned}
 B_{\cR} =(\sum_{i=0}^{p-1}(-1)^{(p-1)i}\tau_{\cR}^{i})\sigma_{\cR}\tau_{\cR} (1-(-1)^{p} \tau_{\cR}) ,
\end{aligned}
\end{align}
with
\begin{align}
\begin{aligned}
& \tau_{\cR}( \alpha \otimes F^0 \otimes \cdots \otimes F^p)=\alpha \otimes F^1 \otimes \cdots \otimes F^p \otimes F^0\\
& \sigma_{\cR}( \alpha \otimes F^0 \otimes \cdots \otimes F^p)=\alpha \otimes F^0 \otimes \cdots \otimes F^{p-1}F^p.
\end{aligned}
\end{align}
\newline
step 3
\newline
Usually we will look at the complex of totally antisymmetric cochains $C^{p,q}_{\cR}(\wedge \Fg_1^{\ast},\wedge \cR(G_2)):=(\wedge^{q} \Fg_1^{\ast}\otimes \wedge^{p+1} \cR(G_2))^{\cR(G_2)}$\\
It is quasi-isomorphic to $C^{p,q}_{\cR}(\wedge \Fg_1^{\ast},\otimes \cR(G_2))$ via the antisymmetrization map $\alpha_{\cR}$
\begin{align*}
\alpha_{\cR} (\alpha \otimes F^0 \wedge \cdots \wedge F^p)= 
\frac{1}{(p+1)!} \sum_{\sigma \in S_{p+1}} (-1)^{\sigma} \alpha \otimes F^{\sigma(0)} \otimes \cdots \otimes F^{\sigma(p)}.
\end{align*}
The quasi-isomorphism is the same as in the case of group cohomology. We refer to~\cite[Appendix~A]{yamagami_polygonal_2002}.
\begin{equation*}
\begin{tikzpicture}[description/.style={fill=white,inner sep=2pt}]
\matrix (m) [matrix of math nodes, row sep=3em, column sep=3em, 
text height=1.5ex, text depth=0.25ex]
{   \vdots & \vdots & \vdots    \\
     \wedge^{2} \Fg_1^{\ast} & (\wedge^{2} \Fg_1^{\ast}\otimes \wedge^{2} \cR)^{\cR} &  (\wedge^{2} \Fg_1^{\ast}\otimes \wedge^{3} \cR)^{\cR} &  \cdots   \\ 
     \Fg_1^{\ast} & (\wedge^{1} \Fg_1^{\ast}\otimes \wedge^{2} \cR)^{\cR} &  (\wedge^{1} \Fg_1^{\ast}\otimes \wedge^{3} \cR)^{\cR} &  \cdots   \\ 
    \bC & (\bC \otimes \wedge^{2} \cR)^{\cR} &  (\bC \otimes \wedge^{3} \cR)^{\cR} &  \cdots   \\  };
  
\path[transform canvas={yshift=0.6ex},->,font=\scriptsize]
(m-2-1) edge node[above] {$ b_{\cR} $} (m-2-2)  (m-2-2) edge node[above] {$ b_{\cR} $} (m-2-3) (m-2-3) edge node[above] {$ b_{\cR} $} (m-2-4)  
(m-3-1) edge node[above] {$ b_{\cR} $} (m-3-2)  (m-3-2) edge node[above] {$ b_{\cR} $} (m-3-3) (m-3-3) edge node[above] {$ b_{\cR} $} (m-3-4)  
(m-4-1) edge  node[above] {$ b_{\cR} $} (m-4-2)  (m-4-2) edge node[above] {$ b_{\cR} $} (m-4-3) (m-4-3) edge node[above] {$ b_{\cR} $} (m-4-4)    ;

\path[transform canvas={xshift=0ex},->,font=\scriptsize]
(m-4-1) edge node[left]{$\partial_{\Fg}$} (m-3-1)  (m-4-2) edge node[left]{$\partial_{\Fg}$} (m-3-2)   (m-4-3) edge node[left]{$\partial_{\Fg}$} (m-3-3)  
(m-3-1) edge node[left]{$\partial_{\Fg}$} (m-2-1)  (m-3-2) edge node[left]{$\partial_{\Fg}$} (m-2-2)   (m-3-3) edge node[left]{$\partial_{\Fg}$} (m-2-3) 
(m-2-1) edge node[left]{$\partial_{\Fg}$} (m-1-1)  (m-2-2) edge node[left]{$\partial_{\Fg}$} (m-1-2)   (m-2-3) edge node[left]{$\partial_{\Fg}$} (m-1-3) ;

\end{tikzpicture} 
\end{equation*}

The boundary operators have simple expression when restricted to such complex.
\begin{align}
\begin{aligned}
&b_{\cR} ( \alpha \otimes F^0 \wedge \dots \wedge F^p)=  \alpha \otimes 1 \wedge F^0 \wedge \dots \wedge F^p, \\
&\partial_{\Fg} (\alpha \otimes F^0 \wedge \dots \wedge F^p)=\partial \alpha \otimes F^0 \wedge \dots \wedge F^p \\
&\qquad \qquad \qquad  \qquad  \qquad -\sum_{i,j}^{}\omega^i \wedge \alpha \otimes F^0 \wedge \dots \wedge Z_i \trt F^j  \wedge \dots \wedge F^p,
\end{aligned}
\end{align}
and because 
\begin{align*}
\tau_{\cR}( \alpha \otimes F^0 \wedge \dots \wedge F^p)= \alpha \otimes F^1 \wedge \dots \wedge F^p \wedge F^0=(-1)^{p} \alpha \otimes F^0 \wedge \dots \wedge F^p,
\end{align*}
we have 
\begin{equation*}
(1-(-1)^{p} \tau_{\cR})( \alpha \otimes F^0 \wedge \dots \wedge F^p)= 0, 
\end{equation*}
therefore
\begin{equation*}
B_{\cR}( \alpha \otimes F^0 \wedge \dots \wedge F^p)= 0.
\end{equation*}

We notice that $C^{p}_{\cR}(G_2,\wedge ^{q} \Fg_1^{\ast})$ is quasi-isomorphic to $C^{p,q}_{\cR}(\wedge \Fg_1^{\ast},\wedge \cR(G_2))$ via the map:
\begin{equation*}
\begin{aligned}
\cJ (\alpha \otimes F^0 \wedge \cdots \wedge F^p) (\psi_0, \dots, \psi_p)= 
\frac{1}{(p+1)!} \sum_{\sigma \in S_{p+1}} (-1)^{\sigma} \alpha  F^{\sigma(0)} (\psi_0) \cdots F^{\sigma(p)} (\psi_p).
\end{aligned}
\end{equation*}

\begin{equation}
\begin{tikzpicture}[description/.style={fill=white,inner sep=2pt}]
\matrix (m) [matrix of math nodes, row sep=3em, column sep=3em, 
text height=1.5ex, text depth=0.25ex]
{   \vdots & \vdots & \vdots    \\
     \wedge^{2} \Fg_1^{\ast} & C^{1}_{\cR}(G_2,\wedge ^{2} \Fg_1^{\ast}) &  C^{2}_{\cR}(G_2,\wedge ^{2} \Fg_1^{\ast}) &  \cdots   \\ 
     \Fg_1^{\ast} & C^{1}_{\cR}(G_2,\wedge ^{1} \Fg_1^{\ast}) &  C^{2}_{\cR}(G_2,\wedge ^{1} \Fg_1^{\ast}) &  \cdots   \\ 
    \bC & C^{1}_{\cR}(G_2,\bC) &  C^{2}_{\cR}(G_2,\bC) &  \cdots   \\  };
  
\path[transform canvas={yshift=0.6ex},->,font=\scriptsize]
(m-2-1) edge node[above] {$ b_{\cR} $} (m-2-2)  (m-2-2) edge node[above] {$ b_{\cR} $} (m-2-3) (m-2-3) edge node[above] {$ b_{\cR} $} (m-2-4)  
(m-3-1) edge node[above] {$ b_{\cR} $} (m-3-2)  (m-3-2) edge node[above] {$ b_{\cR} $} (m-3-3) (m-3-3) edge node[above] {$ b_{\cR} $} (m-3-4)  
(m-4-1) edge  node[above] {$ b_{\cR} $} (m-4-2)  (m-4-2) edge node[above] {$ b_{\cR} $} (m-4-3) (m-4-3) edge node[above] {$ b_{\cR} $} (m-4-4)    ;

\path[transform canvas={xshift=0ex},->,font=\scriptsize]
(m-4-1) edge node[left]{$\partial_{\Fg}$} (m-3-1)  (m-4-2) edge node[left]{$\partial_{\Fg}$} (m-3-2)   (m-4-3) edge node[left]{$\partial_{\Fg}$} (m-3-3)  
(m-3-1) edge node[left]{$\partial_{\Fg}$} (m-2-1)  (m-3-2) edge node[left]{$\partial_{\Fg}$} (m-2-2)   (m-3-3) edge node[left]{$\partial_{\Fg}$} (m-2-3) 
(m-2-1) edge node[left]{$\partial_{\Fg}$} (m-1-1)  (m-2-2) edge node[left]{$\partial_{\Fg}$} (m-1-2)   (m-2-3) edge node[left]{$\partial_{\Fg}$} (m-1-3) ;

\end{tikzpicture} 
\end{equation}

The boundary operators are
\begin{align*}
& b_{\cR} (c)( \psi_0, \dots, \psi_{p+1})= \sum_{i=0}^{p+1}(-1)^i c (\psi_0, \dots, \hat{\psi_{i}}, \dots, \psi_{p+1}) \\
&\partial_{\Fg} (c)(\psi_0, \dots, \psi_{p})= \partial (c( \psi_0, \dots, \psi_{p})) - \sum_{i}^{}\omega^i \wedge ( Z_i \trt c ) (\psi_0, \dots, \psi_{p}).
\end{align*}
\newline
step 4
\newline
we define a map $ \Theta:C^{p}_{\cR}(G_2,\wedge ^{q} \Fg_1^{\ast}) \to D^{p,q}$ 
\begin{align}\label{thetamap}
\Theta(\alpha)(\psi_0,\dots, \psi_p) \vert _{\vp} =L_{\vp}^{\ast}(\alpha) (\psi_0 \tlt \vp,\dots,\psi_{p}\tlt \vp)
\end{align}

\begin{Proposition}
$\Theta$ is an isomorphism between complexes.
\end{Proposition}
\begin{proof}
use the strong covariance property of $D^{p,q}$, cochains in $D^{p,q}$ depend only on its value at $e_1 \in G_1$.
\end{proof}

\subsection{Construction from Lie algebra cohomology complex to \texorpdfstring{$D^{p,q}$}{Dpq}}

In this section, we would like to carry out a parallel construction of a chain map from Lie algebra cohomology complex of the pair $(\Fg_1\bowtie \Fg_2 , \Fh_2)$ to the complex $D^{p,q}$, comparing with the infinite dimensional case as in \cite[2.1]{moscovici_hopf_2011}, and~\cite{connes_hopf_1998}.
\newline
Assume $H_2$ is invariant under $G_1$ action, we can have $H_2 \backslash(G_1\bowtie G_2)  \cong G_1 \bowtie (H_2 \backslash G_2)$ and $\Fh_2  \backslash (\Fg_1\bowtie \Fg_2 ) \cong \Fg_1 \bowtie ( \Fh_2 \backslash\Fg_2  )$. We can also write down right $\Fh_2$-action on $\Fh_2  \backslash (\Fg_1\bowtie \Fg_2 ) \cong \Fg_1 \bowtie ( \Fh_2 \backslash\Fg_2  )$ by the induced adjoint action $(-\trt) \oplus ad$. In order to have a linear action we need to assume that the action $\trt$ of $\Fh_2$ on $\Fg_1$ is given by derivations, this is the case when our matched pair of Lie groups are from decomposition of Lie group. Under above assumptions we can talk about $\Fh_{2}-$basic forms on the bicrossed product group $G_1 \bowtie G_2$.

For $\psi_0,\dots,\psi_{p} \in G_2$, we let $\Delta(\psi_0,\dots,\psi_{p})$ be the affine simplex with vertices $\pi_{L}(\psi_i)$ in the affine coordinates on $L_2$. i.e., 
\begin{align*}
\Delta(\psi_0,\dots,\psi_{p})=\exp \Big(\sum_{0}^{p} t_i \log \big(\pi_{L}(\psi_i)\big)\Big); \qquad 0 \le t_i \le 1,\sum_{0}^{p} t_i=1 
\end{align*}
where $\pi_L$ is the projection $G_2 \to H_2 \backslash G_2$.
We check that the right multiplication of $G_2$ is affine on the exponential coordinates of $L_2$, i.e., assume $\psi=hl$,
\begin{align}
\begin{aligned}
\Delta(\psi_0  \psi,\dots,\psi_{p}  \psi)=&\exp \Big(\sum_{0}^{p} t_i \log \big(\pi_{L}(\psi_i  \psi)\big)\Big)\\
=&\exp \Big(\sum_{0}^{p} t_i \log \big((\pi_{L}(\psi_i) \tlt h  )l)\big)\Big)\\
=&(  \Delta(\psi_0,\dots,\psi_{p})  \tlt h) l  
\end{aligned}
\end{align}
where we identify $L_2$ with $H_2 \backslash G_2$ to make sense of the right $G_2$ action.

We assume that the right $G_1$ action is affine on the exponential coordinates, i.e.,
\begin{align}
\begin{aligned}
\Delta(\psi_0 \tlt \vp,\dots,\psi_{p} \tlt \vp)=&\exp \Big(\sum_{0}^{p} t_i \log \big(\pi_{L}(\psi_i \tlt \vp)\big)\Big)\\
=&\exp \Big(\sum_{0}^{p} t_i \log \big(\pi_{L}(\psi_i)\tlt \vp)\big)\Big)\\
=&  \Delta(\psi_0,\dots,\psi_{p}) \tlt \vp   
\end{aligned}
\end{align}

We define $\mathcal{D}$ as discussed in~\cite[2.1]{moscovici_hopf_2011} and~\cite[p229]{connes_hopf_1998}.

\begin{align}
&\mathcal{D} :(\wedge^{n} (\Fg_{1} \oplus \Fl_{2})^{\ast})^{\Fh_{2}}  \to C^{p,q} \\
&\langle \mathcal{D} (\omega)(\psi_0,\dots,\psi_{p}), \eta \rangle = (-1)^{\frac{q(q+1)}{2}}  \int_{(G_1 \times \Delta (\psi_0,\dots,\psi_{p}))^{-1}}  \pi_1^{\ast} \eta \wedge \tilde{\omega}
\end{align}

\begin{Lemma}
The group cochain $\mathcal{D} (\omega)$ satisfies the strong covariance property:
\begin{align}
\begin{aligned}
\mathcal{D} (\omega) (\psi_0 \tlt \phi,\dots,\psi_{p} \tlt \phi ) =(\phi^{-1} \trt )^{\ast} \mathcal{D} (\omega)
\end{aligned}
\end{align}
where $\phi=\vp\psi \in G_1G_2=G$.
\end{Lemma}
\begin{proof}

\begin{align*}
\langle \mathcal{D} (\omega)(\psi_0\phi,\dots,\psi_{p}\phi), \eta \rangle &=\int_{(G_1 \times \Delta (\psi_0\phi,\dots,\psi_{p}\phi))^{-1}}  \pi_1^{\ast} \eta \wedge \tilde{\omega}\\
&=\int_{L_{\phi^{-1}}(G_1 \times \Delta (\psi_0,\dots,\psi_{p}))^{-1}}  \pi_1^{\ast} \eta \wedge \tilde{\omega}\\
&=\int_{(G_1 \times \Delta (\psi_0,\dots,\psi_{p}))^{-1}} (\widetilde{\phi^{}})^{\ast} (\pi_1^{\ast} \eta \wedge \tilde{\omega})\\
&=\int_{(G_1 \times \Delta (\psi_0,\dots,\psi_{p}))^{-1}}  \pi_1^{\ast} ((\phi \trt)^{\ast} \eta) \wedge \tilde{\omega}\\
&=\langle \mathcal{D} (\omega)(\psi_0,\dots,\psi_{p}), (\phi \trt)^{\ast} \eta \rangle \\
&=\langle (\phi^{-1} \trt)^{\ast} \mathcal{D} (\omega)(\psi_0,\dots,\psi_{p}),  \eta \rangle
\end{align*}

\end{proof}

\begin{Proposition} \label{Dchainmap}
$\mathcal{D} $ is a map of complexes
\end{Proposition}
\begin{proof}
The covariance property is checked above. Now we check 
\begin{align}
\begin{aligned}
\mathcal{D} (d\omega) =(d_1+d_2) \mathcal{D} (\omega)
\end{aligned}
\end{align}
we have
\begin{align}
\begin{aligned}
&\langle \mathcal{D} (d\omega)(\psi_0,\dots,\psi_{p}), \eta \rangle \\
=& (-1)^{\frac{q(q+1)}{2}}  \int_{(G_1 \times \Delta (\psi_0,\dots,\psi_{p}))^{-1}}  \pi_1^{\ast} \eta \wedge d\tilde{\omega} \\
=& (-1)^{\frac{q(q+1)}{2}}   (-1)^{q} ( \int_{(G_1 \times \Delta (\psi_0,\dots,\psi_{p}))^{-1}}  d(\pi_1^{\ast} \eta \wedge \tilde{\omega}  )      \\
&  - \int_{(G_1 \times \Delta (\psi_0,\dots,\psi_{p}))^{-1}}  \pi_1^{\ast} (d\eta) \wedge  \tilde{\omega}      ) \\
=&  (-1)^{q} (-1)^{\frac{q(q+1)}{2}}  ( \int_{(G_1 \times \Delta (\psi_0,\dots,\psi_{p}))^{-1}}  d(\pi_1^{\ast} \eta \wedge \tilde{\omega}  )  )     \\
& +(-1)^{\frac{(q+1)(q+2)}{2}} ( \int_{(G_1 \times \Delta (\psi_0,\dots,\psi_{p}))^{-1}}  \pi_1^{\ast} (d\eta) \wedge  \tilde{\omega}      ) \\
=&  (-1)^{q} (-1)^{\frac{q(q+1)}{2}}  ( \int_{(G_1 \times \Delta (\psi_0,\dots,\psi_{p}))^{-1}}  d(\pi_1^{\ast} \eta \wedge \tilde{\omega}  )  )     + \langle d_2 \mathcal{D} (\omega), \eta \rangle.
\end{aligned}
\end{align}
Apply Stokes, the first integral becomes the group coboundary:
\begin{align}
\begin{aligned}
 &(-1)^{q} (-1)^{\frac{q(q+1)}{2}}  ( \int_{(G_1 \times \Delta (\psi_0,\dots,\psi_{p}))^{-1}}  d(\pi_1^{\ast} \eta \wedge \tilde{\omega}  )  ) \\
 =& (-1)^{q} (-1)^{\frac{q(q+1)}{2}} (\sum_{i=0}^{p+1} (-1)^{i}    \int_{(G_1 \times \Delta (\psi_0,\dots,\hat{\psi_{i}},\dots, \psi_{p}))^{-1}}  \pi_1^{\ast} \eta \wedge \tilde{\omega}   ) \\
 =&\sum_{i=0}^{p+1} (-1)^{q+i}  \langle \mathcal{D} (\omega)(\psi_0,\dots,\hat{\psi_{i}},\dots, \psi_{p}), \eta \rangle \\
 =&\langle d_1 \mathcal{D} (\omega) ,\eta \rangle.
\end{aligned}
\end{align}

\end{proof}

Because of the strong covariance property, $D$ is completely determined by its value at $e_1 \in G_1$, consider group cochains with values in $\wedge ^{\bullet} \Fg_{1}^{\ast}$:
\begin{align}
\begin{aligned}
\mathcal{E}(\omega)(\psi_0,\dots,\psi_{p}):=\mathcal{D}(\omega)(\psi_0,\dots,\psi_{p}) \vert _{e_1}.
\end{aligned}
\end{align}
\begin{Corollary} \label{Echain1}
$\mathcal{E}$ is a chain map from $(\wedge^{n} (\Fg_{1} \oplus \Fl_{2})^{\ast})^{\Fh_{2}}$ to $C^{p}(G_2,\wedge ^{q} \Fg_1^{\ast}) $
\end{Corollary}
\begin{proof}
The covariance property is preserved by the evaluation. \\
Use Prop.~\ref{Dchainmap}, the evaluation of the second integral in its proof at $e_1$ is clearly group coboundary, while the first integral, when evaluate at $e_1$, gives Lie algebra coboundary with coefficients.
\end{proof}
As discussed in~\cite[2.1]{moscovici_hopf_2011} we can also verify the explicit description of $\mathcal{E}$:
Fix a basis $Z_i$ of $\Fg_1$, denote by $\tilde{Z_i}$ the corresponding left invariant vector fields, $\omega_i$ the dual basis, $\tilde{\omega}_i$ the corresponding left invariant forms fields. Use the formula in~\cite[2.2]{moscovici_hopf_2011}, we define $\nu(\vp,\psi)=\vp(\psi \tlt \vp)^{-1}$. Let $\imath:L_2 \to L_2$ be the inversion map, $\iota$ be the contraction. Let $\omega$ be in $(\wedge^{n} (\Fg_{1} \oplus \Fl_{2})^{\ast})^{\Fh_{2}}$ and $\tilde{\omega}$ the corresponding left invariant $\Fh_2-$basic form on $(G_1\bowtie G_2)/H_2$. It is the key that we can have explicit expression of such invariant form in order to make use of this explicit $\mathcal{E}$ map. Denote 
\begin{align}
\begin{aligned}\label{mumap}
&\mu :  (\wedge^{n} (\Fg_{1} \oplus \Fl_{2})^{\ast})^{\Fh_{2}} \to \sum_{p+q=n} (\Omega^{p}(L_2) )^{\Fh_{2}} \otimes \wedge^{q} \Fg_1^{\ast}\\
&\mu_q(\omega)=\sum_{|I|=q} \imath^\ast(\iota_{\tilde{Z}_{I}(e)}\nu^\ast(\tilde{\omega})) \otimes \omega_{I}, \\
& I=(i_1 < \dots < i_q)\quad \text{and} \quad \omega_{I}=\omega_{i_1}\wedge \cdots \wedge \omega_{i_q}
\end{aligned}
\end{align}
Use the formula in~\cite[2.2]{moscovici_hopf_2011} We then write a map $\mathcal{E}=\int\limits_{\Delta} \circ \mu $ from $(\wedge^{n} (\Fg_{1} \oplus \Fl_{2})^{\ast})^{\Fh_{2}}$ to $ \sum_{p+q=n}  C^{p}_{\cR}(G_2,\wedge ^{q} \Fg_1^{\ast})$:
\begin{align} \label{mapE}
\mathcal{E} (\omega) (\psi_0,\dots,\psi_{p})=  \int\limits_{\Delta(\psi_0,\dots,\psi_{p})} \mu_q(\omega)
\end{align}
one can also compare this map with the integration along simplex map as used by Dupont in~\cite{dupont_simplicial_1976}. 
Then we can have 
\begin{align} 
\mathcal{E} (\omega) (\psi_0,\dots,\psi_{p})=  \mathcal{D} (\omega) (\psi_0,\dots,\psi_{p}) \vert _{e_1}
\end{align}

\begin{Proposition}
$\mathcal{E} (\omega)$ are representative functions of $\psi_i$s.
\end{Proposition}
\begin{proof}
Because $L_2=\exp(\Fl_2)$ is nilpotent, representative functions are exactly the polynomial functions. Use exponential coordinates on $L_2$ and we can represent the projection of $ \nu^\ast(\tilde{\omega})$ on $L_2$ as forms with polynomial coefficients. Now apply the contractions $\iota_{\tilde{Z}_{e}^{I}}$, the only new type of coefficients that appear are of the form $Z_I \trt P$ where $P$ is some polynomial function of $L_2$. By Lemma~\ref{11} $Z_I \trt P$ is again representative hence polynomial. Therefore $ \imath^\ast(\iota_{\tilde{Z}_{I}(e)}\nu^\ast(\tilde{\omega})) $ are polynomial differential forms on $L_2$ and $\mathcal{E} (\omega)$ are representative functions on $L_2$. 

Furthermore they are $H_2$ invariant, every $H_2$ invariant $L_2$ representative function is representative on $G_2$ because of the semi direct product decomposition $G_2=H_2 \ltimes L_2$. More precisely if we extend any $F \in \cR(L_2)$ trivially to a function on $G_2$ by $F(hl)=F(l)$, then

\begin{align*}
F \big( (h_1l_1)(h_2l_2) \big)&=F \big( (h_1h_2)(l_1 \tlt h_2 l_2 ) \big) =F(l_1 \tlt h_2 l_2 )=F_{(1)}(l_1 \tlt h_2) F_{(2)}(l_2) \\
&=F_{(1)}( l_1) F_{(2)}(l_2) =F_{(1)}( h_1l_1) F_{(2)}(h_2l_2),
\end{align*}

hence $F \in  \cR(G_2) $ and $\pi_L^{\ast}$ preserves representative functions.

Therefore $\mathcal{E} (\omega)$ are representative functions on $G_2$. 
\end{proof}

\begin{Corollary}
$\mathcal{E}$ is a chain map from $(\wedge^{n} (\Fg_{1} \oplus \Fl_{2})^{\ast})^{\Fh_{2}}$ to $C^{p}_{\cR}(G_2,\wedge ^{q} \Fg_1^{\ast}) $
\end{Corollary}

\begin{proof}
Use previous proposition and corollary~\ref{Echain1}.
\end{proof}

\begin{Corollary}
$\mathcal{D}$ is a chain map from $(\wedge^{n} (\Fg_{1} \oplus \Fl_{2})^{\ast})^{\Fh_{2}}$ to $D^{p,q} \subset C^{p,q}$
\end{Corollary}

\begin{proof}

The map $\mathcal{E}$ can be extended by the isomorphism $\Theta:C^{p}_{\cR}(G_2,\wedge ^{q} \Fg_1^{\ast}) \to D^{p,q}$ defined in~\ref{thetamap}
\begin{align*}
\Theta(\alpha)(\psi_0,\dots, \psi_p) \vert _{\vp} =L_{\vp}^{\ast}(\alpha) (\psi_0 \tlt \vp,\dots,\psi_{p}\tlt \vp)
\end{align*}
back to map $\mathcal{D}=\Theta \circ \mathcal{E}$:
\begin{align*}
\mathcal{D}(\omega) (\psi_0,\dots,\psi_{p})\vert_{\vp}=L_{\vp}^{\ast}(\mathcal{E} (\omega)) (\psi_0 \tlt \vp,\dots,\psi_{p}\tlt \vp)
\end{align*}

\end{proof}

We can also have the identification below in this case:
$(\wedge^{n} (\Fg_{1} \oplus \Fl_{2})^{\ast})^{\Fh_{2}} \cong \bigoplus_{p+q=n}(\wedge^{q} \Fg_1^{\ast}\otimes \wedge^{p} (\Fl_2^{\ast})^{\Fh_2})$ \label{identification2}
follow the proof in~\cite[Lemma 2.7]{moscovici_hopf_2011} the isomorphism is given by
\begin{align}
\begin{aligned}
&\natural: (\wedge^{n} (\Fg_{1} \oplus \Fl_{2})^{\ast})^{\Fh_{2}} \to \bigoplus_{p+q=n}(\wedge^{q} \Fg_1^{\ast}\otimes \wedge^{p} (\Fl_2^{\ast})^{\Fh_2})\\
&\natural (\omega)(X^1,\dots, X^q \vert \xi_1,\dots \xi_p)= \omega (X^1\oplus 0, \dots , X^q \oplus 0, 0 \oplus \xi_1, \dots , 0 \oplus \xi_p)
\end{aligned}
\end{align}
with its inverse 
\begin{align}
\begin{aligned}
&\natural^{-1} (\mu \otimes \nu)(X^1 \oplus \xi_1 , \dots , X^{p+q} \oplus \xi_{p+q})\\
=&\sum_{\sigma \in Sh(p,q)}(-1)^{\sigma} \mu (X^{\sigma(1)},\dots,X^{\sigma(p)}) \nu (\xi _{\sigma(p+1)},\dots, \xi _{\sigma(p+q)} )
\end{aligned}
\end{align}

Use a slight modification of \cite[Lemma.~2.8,Cor.~2.9]{moscovici_hopf_2011} we can show that $\mu \vert _{e_2} $ coincides with $\natural$.

For $\xi_1,\dots \xi_p \in \Fl_2$, denote 
\begin{equation*}
\begin{aligned}
&\Psi(\xi_1,\dots \xi_p)=\\
&\big(1,\exp(s_{\sigma(1)}\xi_{\sigma(1)}),\exp(s_{\sigma(1)}\xi_{\sigma(1)})\exp(s_{\sigma(2)}\xi_{\sigma(2)}),\dots,\exp(s_{\sigma(1)}\xi_{\sigma(1)})\cdots\exp(s_{\sigma(p)}\xi_{\sigma(p)})\big)
\end{aligned}
\end{equation*}
Define a map $j:C^{p}_{\cR}(G_2,\wedge ^{q} \Fg_1^{\ast}) \to \wedge^{q} \Fg_1^{\ast}\otimes \wedge^{p} (\Fl_2^{\ast})^{\Fh_2}$
\begin{align*}
 &j(\alpha)(X^1,\dots, X^q \vert \xi_1,\dots \xi_p) \\
 =&\sum_{\sigma \in S_{p}}  (-1)^{\sigma} \left.\frac{d}{ds_1}\right|_{_{s_1=0}} \cdots \left.\frac{d}{ds_p}\right|_{_{s_p=0}} \alpha(\Psi(\xi_1,\dots \xi_p))(X^1,\dots, X^q)
\end{align*}

then The map $\mathcal{E}$ has a quasi-inverse $ \natural ^{-1}\circ j$.\\

\begin{Lemma}
$( \natural ^{-1}\circ j) \circ \mathcal{E} $ is identity. 
\end{Lemma}

\begin{proof}
since  $\mu \vert _{e_2}$ implement the $\natural$ map, we just need to show that $j\circ \int\limits_{\Delta}$ is the evaluation at $e_2$ .
we write $\xi_{k}=\xi_k^{j}\frac{\partial}{\partial x_{j}}$ and
write any $p$ form as $f_{i_1\dots i_p} dx^{i_1} \wedge \cdots \wedge  dx^{i_p}$ in the exponential coordinates.
\begin{equation*}
\begin{aligned}
x^{i_k}=&t_1 ( s_{\sigma(1)} \xi_{\sigma(1)}^{i_k}) +t_2 (s_{\sigma(1)}\xi_{\sigma(1)}^{i_k}+s_{\sigma(2)}\xi_{\sigma(2)}^{i_k} +p_2(s_{\sigma(1)},s_{\sigma(2)})) +\cdots \\
&+ t_k (s_{\sigma(1)}\xi_{\sigma(1)}^{i_k}+ \cdots + s_{\sigma(k)}\xi_{\sigma(k)}^{i_k} +p_{\sigma(k)}(s_{\sigma(1)},s_{\sigma(2)},\dots s_{\sigma(k)})) , \quad 1 \le k \le p 
\end{aligned}
\end{equation*}
where $p_2,\dots, p_k$ are $\xi-$valued polynomials with degree higher than $2$. 
\begin{equation*}
\begin{aligned}
d(x^{i_k})=& ( s_{\sigma(1)} \xi_{\sigma(1)}^{i_k}) d t_1  + (s_{\sigma(1)}\xi_{\sigma(1)}^{i_k}+s_{\sigma(2)}\xi_{\sigma(2)}^{i_k} +p_2(s_{\sigma(1)},s_{\sigma(2)})) d t_2 +\cdots \\
 &+ (s_{\sigma(1)}\xi_{\sigma(1)}^{i_k}+ \cdots + s_{\sigma(k)}\xi_{\sigma(k)}^{i_k} +p_{\sigma(k)}(s_{\sigma(1)},s_{\sigma(2)},\dots s_{\sigma(k)})) d t_k  , \quad 1 \le k \le p 
\end{aligned}
\end{equation*}
therefore if we calculate the wedge product of $dx^{i_1} \wedge \cdots \wedge  dx^{i_p}$, its degree will be at least $p$, viewed as $\xi dt^{1} \wedge \cdots \wedge  dt^{p}-$valued polynomials in $s_j$. Since we want to take derivative from $s_1$ to $s_p$, there is exact one term that gives non-zero value:

\begin{align*}
&\sum_{\sigma \in S_{p}}  (-1)^{\sigma} \left.\frac{d}{ds_1}\right|_{_{s_1=0}} \cdots \left.\frac{d}{ds_p}\right|_{_{s_p=0}}  \\
&\int\limits_{\Delta(\Psi(\xi_1,\dots \xi_p))} \sum_{i_1 < \dots < i_p} f_{i_1\dots i_p} dx^{i_1} \wedge \cdots \wedge  dx^{i_p} \\
=&\sum_{\sigma \in S_{p}}  (-1)^{\sigma} \left.\frac{d}{ds_1}\right|_{_{s_1=0}} \cdots \left.\frac{d}{ds_p}\right|_{_{s_p=0}} \\
&\int\limits_{\Delta^p} \sum_{i_1 < \dots < i_p} f_{i_1\dots i_p} s_{\sigma(1)} \xi_{\sigma(1)}^{i_1}dt_1\wedge \cdots \wedge  s_{\sigma(p)}\xi_{\sigma(p)}^{i_p} dt_{p} \\
=&\sum_{\sigma \in S_{p}}  (-1)^{\sigma} \int\limits_{\Delta^p} \sum_{i_1 < \dots < i_p} f_{i_1\dots i_p} (e_2)  \xi_{\sigma(1)}^{i_1}(e_2)dt_1\wedge \cdots \wedge  \xi_{\sigma(p)}^{i_p} (e_2) dt_{p} \\
=&\sum_{\sigma \in S_{p}}  (-1)^{\sigma}  \frac{1}{p! } \sum_{i_1 < \dots < i_p}( f_{i_1\dots i_p}  \xi_{\sigma(1)}^{i_1} \cdots  \xi_{\sigma(p)}^{i_p} )(e_2) \\
=&(f_{i_1\dots i_p} dx^{i_1} \wedge \cdots \wedge  dx^{i_p})(\xi_1^{j}\frac{\partial}{\partial x_{j}}, \dots , \xi_p^{j}\frac{\partial}{\partial x_{j}}) \vert_{e_2}
\end{align*}

\end{proof}

\begin{Proposition}\label{Equasi}
$\mathcal{E}$ is a quasi-isomorphism. 
\end{Proposition}
\begin{proof}
In~\cite[Theorem~4.10]{rangipour_van_2012} the authors showed that $(\wedge^{n} (\Fg_{1} \oplus \Fl_{2})^{\ast})^{\Fh_{2}}$ is quasi-isomorphic to the double complex $C^{q,p}(\wedge \Fg_1^{\ast},\otimes \cR(G_2))$ (\ref{doublecomplex2}). The later is also quasi-isomorphic to $C^{p}_{\cR}(G_2,\wedge ^{q} \Fg_1^{\ast})$. Therefore $\mathcal{E}$ as a chain map with one side inverse between two quasi-isomorphic complexes is a quasi-isomorphism.
\end{proof}
we can also give a direct proof here to show that $C^{p}_{\cR}(G_2,\wedge ^{q} \Fg_1^{\ast})$ is quasi-isomorphic to $(\wedge^{n} (\Fg_{1} \oplus \Fl_{2})^{\ast})^{\Fh_{2}}$.
\begin{proof}
we will take a look at the triple complex $C^{s}_{\cR}(G_2,\Omega_{L_2}^{t} \otimes \wedge ^{q} \Fg_1^{\ast})$. During the first step we will fix $q$ and vary $s$, $t$ and view it as a double complex for each fixed $q$.
we put an augmented row
\begin{align*}
C^{s}_{\cR}(G_2,\bC \otimes \wedge ^{q} \Fg_1^{\ast}) 
\end{align*}
there is an inclusion
\begin{align*}
0 \to C^{s}_{\cR}(G_2,\bC \otimes \wedge ^{q} \Fg_1^{\ast})  \xrightarrow{\varepsilon_1} C^{s}_{\cR}(G_2,\Omega_{L_2}^{0} \otimes \wedge ^{q} \Fg_1^{\ast}) \to C^{s}_{\cR}(G_2,\Omega_{L_2}^{1} \otimes \wedge ^{q} \Fg_1^{\ast}) \to \cdots
\end{align*}
and an augmented column
\begin{align*}
(\Omega_{L_2}^{t})^{G_2} \otimes \wedge ^{q} \Fg_1^{\ast}
\end{align*}
there is another inclusion
\begin{align*}
0 \to (\Omega_{L_2}^{t})^{G_2} \otimes \wedge ^{q} \Fg_1^{\ast}  \xrightarrow{\varepsilon_2} C^{0}_{\cR}(G_2,\Omega_{L_2}^{t} \otimes \wedge ^{q} \Fg_1^{\ast}) \to C^{1}_{\cR}(G_2,\Omega_{L_2}^{t} \otimes \wedge ^{q} \Fg_1^{\ast}) \to \cdots
\end{align*}
the columns are exact for $s \ge 0$ since $L$ is smoothly contractible.
the rows are exact for $t \ge 0$ because of the homotopy below (cf.,~\cite[Lemma~2.3]{kumar_extensions_2006}):
\begin{align*}
&H:\ C^{s}_{\cR}(G_2,\Omega_{L_2}^{t} \otimes \wedge ^{q} \Fg_1^{\ast}) \to \ C^{s-1}_{\cR}(G_2,\Omega_{L_2}^{t} \otimes \wedge ^{q} \Fg_1^{\ast})  \\
&H(c)(\psi_1,\dots,\psi_{p})=\pi(f_{(\psi_1,\dots,\psi_{p})}^c)
\end{align*}
where
\begin{align*}
f_{(\psi_1,\dots,\psi_{p})}^c (h):=c(hl,\psi_1,\dots,\psi_{p}) (l) \qquad  l\,\in L_2, h\,\in H_2
\end{align*}
and $\pi$ the projection to its $H_2$ invariant part, since $H_2$ is reductive and any $H_2$ module is completely reducible.
By the standard diagram chasing of double complex we have that $\varepsilon_1$ and $\varepsilon_2$ both induces isomorphisms to the cohomology of the double complex. Therefore these augmented complexes are quasi-isomorphic, while the augmented column is isomorphic to the complex $(\wedge^{n} (\Fg_{1} \oplus \Fl_{2})^{\ast})^{\Fh_{2}}$. 

\end{proof}

\subsection{Main theorem}

\begin{Lemma} \label{Hoch-iso}
The map $\Phi_{D} \circ \Theta$ from $(\sum_{p+q=n}C^{p}_{\cR}(G_2,\wedge ^{q} \Fg_1^{\ast}),d_1)$ to the Hochschild complex of $\cH$ is a quasi-isomorphism. 
\end{Lemma}

\begin{proof}
Recall from Corollary~\ref{representativecochain2}, we shall identify $\cH^{\natural}_{(\delta,\sigma^{})}$ and $\text{Im}(\lambda^{\natural})$. In order to prove the statement we need to understand the Hochschild cohomology of the algebra $\cA$ with coefficients in the module $\bC$ given by the augmentation $\varepsilon$ on $\cA$. 

We use the abstract version of the Hochschild-Serre spectral sequence.

The following is from~\cite[Thm~6.1,p.~349]{cartan_homological_1956} 

A subalgebra $\cA_1 \subset \cA$ of an augmented algebra $\cA$ is called normal if the left ideal $J$ generated by $\text{Ker}\, \varepsilon_1$, where $\varepsilon_1=\varepsilon \vert _{\cA_1}$, is also a right ideal. 

we let 
\begin{align}
\cA_{2}=\cA / J,\qquad \text{or equivalently}\quad \cA_{2}=\bC \otimes_{\cA_1} \cA,
\end{align}
and then we have a spectral sequence converging to the Hochschild cohomology $H^{\ast}_{\cA}(\bC)$ of augmented algebra $\cA$ with coefficients in $\bC$ and with $E_2$ term given by 
\begin{align*}
H^{p}_{\cA_2}(H^{q}_{\cA_1}(\bC)).
\end{align*}

In our case we let $\cA=C_c^{\ify}(G_1)\rtimes G_2^{\delta}$ and $\cA_1=C_c^{\ify}(G_1)$, with augmentation 
\begin{align*}
\varepsilon(fU^{\ast}_{\psi})=f(1)\qquad \forall f \in C_c^{\ify}(G_1),\ \psi \in G_2.
\end{align*}
Thus the ideal $J$ is generated by 
\begin{align*}
g U^{\ast}_{\psi},\qquad g(1)=0.
\end{align*}
It is a two sided ideal because $\forall \psi \in G_2$, $\tilde{\psi}(1)=1$ and 
\begin{align*}
g(1)=0\qquad\text{iff}\quad (g \circ \tilde{\psi})(1)=0
\end{align*}
therefore the algebra $\cA_{2}=\bC \otimes_{\cA_1} \cA$ is the group ring of $G_2$. 

Now since the original version of Hochschild-Serre spectral sequence is for unital algebra, in order to apply the spectral sequence we need one more preparation, i.e., we would like to extend the result to our algebra. 

A algebra $\cA$ is said to have local units if for every finite family of elements $a_i \in \cA $ there is an element $u \in \cA$ such that $ua_i=a_iu=a_i$ for all $i$.

If an algebra has local units, it is excisive for Hochschild cohomology.~\cite[p32]{loday_cyclic_1997}\cite{wodzicki_excision_1989}.

Now we let $\cA=C_c^{\ify}(G_1)\rtimes G_2^{\delta}$ and $\cA_1=C_c^{\ify}(G_1)$. It is obvious that both our algebras $\cA_1$ and $\cA$ have local units (just take $\text{Id} \vert _{\bigcup \text{supp}(f_i)} U^{\ast}_{e}$ and $\text{Id} \vert _{\bigcup \text{supp}(f_i)}$ ). Hence we add a unit to $\cA_1$ as usual:
\begin{align*}
0 \to \bC  \to \tilde{\cA_1}  \to \cA_1\to 0
\end{align*}
where $\tilde{\cA_1}=\bC \oplus \cA_1$ with multiplication $(\lambda,u)(\mu,v)=(\lambda\mu,\lambda v+u \mu + uv)$ and unit $(1,0)$,
and add a unit to $\cA$:
\begin{align*}
0 \to \bC \cA_2  \to \tilde{\cA}  \to \cA\to 0
\end{align*}
where $\tilde{\cA}=\bC \cA_2 \oplus \cA$ with multiplication $(\lambda,u)(\mu,v)=(\lambda\mu,\lambda v+u \mu + uv)$ and unit $(1e,0)$.

Thus we can have a diagram as below:

\begin{equation*}
\begin{tikzpicture}[description/.style={fill=white,inner sep=2pt}]
\matrix (m) [matrix of math nodes, row sep=3em, column sep=3em, 
text height=1.5ex, text depth=0.25ex]
{ \dots & H^{p}_{\cA_2}(H^{q}_{\cA_1}(\bC)) & H^{p}_{\cA_2}(H^{q}_{\tilde{\cA_1}}(\bC)) & H^{p}_{\cA_2}(H^{q}_{\bC}(\bC)) & \dots \\ 
\dots & H^{p+q}_{\cA}(\bC)  & H^{p+q}_{\tilde{\cA}}(\bC) & H^{p+q}_{\cA_2}(\bC) & \dots  \\   };
  
\path[transform canvas={yshift=0ex},->,font=\scriptsize]
(m-1-1) edge  (m-1-2)  (m-1-2) edge   (m-1-3) (m-1-3) edge  (m-1-4)  (m-1-4) edge  (m-1-5)
(m-2-1) edge   (m-2-2)  (m-2-2) edge   (m-2-3) (m-2-3) edge   (m-2-4)  (m-2-4) edge   (m-2-5);

\path[transform canvas={xshift=0ex},->,font=\scriptsize]
(m-1-2) edge  (m-2-2)  (m-1-3) edge  (m-2-3)   (m-1-4) edge  (m-2-4)  ;

\end{tikzpicture} 
\end{equation*}
the first row is due to the excision property of $\cA_1$ and the second row is due to the excision property of $\cA$. Now we apply the abstract Hochschild-Serre spectral sequence for the middle and right columns and conclude that $H^{p}_{\cA_2}(H^{q}_{\cA_1}(\bC))$ also converge to $H^{p+q}_{\cA}(\bC)$. This proved the Hochschild-Serre spectral sequence for our algebra $\cA$.

Now $\text{Im}(\lambda^{\natural})$ can be teated as $\cA$ with the weak topology given by the range of $\lambda$, we can have a corresponding spectral sequence which converges to the Hochschild cohomology of $\cH$ and whose $E_2$ term is given by the representative group cohomology of $G_2$ with coefficients in $H^{\ast}_{\cA_1}(\bC)$, which is given by $\wedge^{q} \Fg_{1}^{\ast}$. On $E_2$ level, the map $\Phi_{D}  \circ \Theta$ realize the representative group cochains of $G_2$ with coefficients in $\wedge^{q} \Fg_{1}^{\ast}$.

\end{proof}

We now have the main result:

\begin{Theorem}\label{mthm1}
Let ($G_1$, $G_2$) be a matched pair of Lie groups, and $L_2$ be a nucleus of $G_2$ that is nilpotent. Let $\Fh_2$, $\Fg_1$, $\Fg_2$ and $\Fl_2$ denote the Lie algebras of $H_2:= G_2/L_2$, $G_1$, $G_2$ and $L_2$ respectively. Let us also assume that $H_2$ is $G_1-$invariant, the action $\trt$ of $\Fh_2$ on $\Fg_1$ is given by derivations, and the right $G_1$ action on $L_2$ is affine in its exponential coordinates. Then we have
$ \Phi_{D}  \circ \mathcal{D} $ induces an isomorphism from the relative Lie cohomology of the pair $(\Fg_{1} \bowtie \Fg_{2}, \Fh_{2})$ to the periodic cyclic cohomology of $\cH=(\cU(\Fg_1)   \acr  \cR(G_2))^{\cop}$ with coefficient $^\sigma \bC _ \delta$, shifted by $\text{dim}(G_1)$. We have 
\begin{align}
\begin{aligned}
&\bigoplus_{i\,\text{\tiny even} } H^{i} (\Fg_{1} \bowtie \Fg_{2}, \Fh_{2},\Cb) \cong HP^{\text{\tiny even}} (\cH,^\sigma \bC _ \delta) \\
&\bigoplus_{i\,\text{\tiny odd} } H^{i} (\Fg_{1} \bowtie \Fg_{2}, \Fh_{2},\Cb) \cong HP^{\text{\tiny odd}} (\cH,^\sigma \bC _ \delta)
\end{aligned}
\end{align}
if $\text{dim}(G_1)$ is even, and 
\begin{align}
\begin{aligned}
&\bigoplus_{i\,\text{\tiny even} } H^{i} (\Fg_{1} \bowtie \Fg_{2}, \Fh_{2},\Cb) \cong HP^{\text{\tiny odd}} (\cH,^\sigma \bC _ \delta) \\
&\bigoplus_{i\,\text{\tiny odd} } H^{i} (\Fg_{1} \bowtie \Fg_{2}, \Fh_{2},\Cb) \cong HP^{\text{\tiny even}} (\cH,^\sigma \bC _ \delta)
\end{aligned}
\end{align}
if $\text{dim}(G_1)$ is odd.
\end{Theorem}

\begin{proof}
$ \Phi_{D}  \circ \Theta $ is a map between complexes $\{ \sum_{}C^{p}_{\cR}(G_2,\wedge ^{q} \Fg_1^{\ast}),d_1,d_2 \}$ and $\{CC^{\bullet}(\cH;^\sigma \bC _ \delta), b, B\}$. By Lemma~\ref{Hoch-iso}, The map $ \Phi_{D}  \circ \Theta $ induces an isomorphism on Hochschild cohomology, therefore it also induces an isomorphism on cyclic cohomology. By Proposition~\ref{Equasi}, $\mathcal{E}$ induces isomorphism between relative Lie cohomology of the pair $(\Fg_{1} \bowtie \Fg_{2}, \Fh_{2})$ and the cohomology of $(\sum_{p+q=n}C^{p}_{\cR}(G_2,\wedge ^{q} \Fg_1^{\ast}),d_1,d_2)$. Therefore the map $ \Phi_{D}  \circ \mathcal{D} = \Phi_{D}  \circ \Theta  \circ \mathcal{E}$ induces an isomorphism from the relative Lie cohomology of the pair $(\Fg_{1} \bowtie \Fg_{2}, \Fh_{2})$ to the periodic cyclic cohomology of $\cH=(\cU(\Fg_1)   \acr  \cR(G_2))^{\cop}$ with coefficient $^\sigma \bC _ \delta$.
The shift by the number $\text{dim}(G_1)$ is included in the formulation of $\Phi_{D}$ map.

\end{proof}

\section{Example construction and calculation}
\label{sec:examples}

\subsection{Real diamond group \texorpdfstring{$\Rb \ltimes H_{3}$}{D4} }

The so called real diamond Lie algebra is the 4-dimensional solvable Lie algebra $\Fd$ with basis $T$, $X$, $Y$, $Z$ satisfying the following commutation relations, see~\cite{brezin_unitary_1968,son_sur_1984}:
\begin{align*}
[T,X]=-Y,[T,Y]=X,[X,Y]=Z,
\end{align*}
this real diamond Lie algebra $\Fd=\Rb \ltimes \Fh_3$ is an extension of the one-dimensional Lie algebra $\Rb T$ by the Heisenberg algebra $\Fh_3$ with basis $X$, $Y$, $Z$.

We look at the corresponding Lie group $\Rb \ltimes H_{3}$. A group element is represented as $(\theta,x+\mathbf{i}y,z)$, while the group multiplication is given by :
\begin{equation*}
\begin{aligned}
&(\theta_1,x_1+\mathbf{i}y_1,z_1)\cdot(\theta_2,x_2+\mathbf{i}y_2,z_2) =\\
&\big(\theta_1+\theta_2,e^{-\mathbf{i}\theta_2}(x_1+\mathbf{i}y_1)+(x_2+\mathbf{i}y_2),z_1+z_2-\frac{1}{2}Im((x_1+\mathbf{i}y_1) e^{-\mathbf{i}\theta_2} (x_2-\mathbf{i}y_2) )\big)
\end{aligned}
\end{equation*}
and inverse is given by
\begin{align}
(\theta,x+\mathbf{i}y,z)^{-1}=(-\theta,-e^{\mathbf{i}\theta}(x+\mathbf{i}y),-z)
\end{align}
when $\theta=0$ we have the usual $H_{3}$. Denote $G_1=(\theta,0,0)$, $G_2=(0,x+\mathbf{i}y,z)$, we have
\begin{align}
(\theta,x+\mathbf{i}y,z)=(\theta,0,0)\cdot(0,x+\mathbf{i}y,z),\qquad \text{i.e.,} \quad G=G_1G_2.
\end{align}
Because
\begin{align}
(0,x+\mathbf{i}y,z_1) \cdot (\theta,0,0)= (\theta, e^{-\mathbf{i}\theta} (x+\mathbf{i}y) , z )=(\theta,0,0)\cdot(0,e^{-\mathbf{i}\theta} (x+\mathbf{i}y),z)
\end{align}
we have 
\begin{align}
\begin{aligned}
&G_2 \trt G_1 :(0,x+\mathbf{i}y,z) \trt  (\theta,0,0) = (\theta,0,0), \\
&G_2 \tlt G_1 : (0,x+\mathbf{i}y,z) \tlt  (\theta,0,0) = (0,e^{-\mathbf{i}\theta} (x+\mathbf{i}y),z)
\end{aligned}
\end{align}
Now we look at the Lie algebra:

\begin{align*}
\Fd \, = \,\{  T,X,Y,Z  \}= \,\{  e_1 ,e_2, e_3 ,e_4  \} \qquad
 \Fg_2= \,  \{ X,Y,Z\} \qquad
 \Fg_1= \,  \{ T \}.
\end{align*}
with Lie bracket given by
\begin{align*}
[T,X]=-Y,[T,Y]=X,[X,Y]=Z
\end{align*}
and actions:
\begin{align}
 \Fg_2 \trt  \Fg_1 :   \text{trivial action}  \qquad  \Fg_2 \tlt  \Fg_1: X \tlt T =Y, Y \tlt T =-X.
\end{align}

The cohomology ring of $H^{\ast}( \Fd )$ is generated by $\theta_1$, $\theta_2\wedge \theta_3\wedge \theta_4$ and $\theta_1\wedge \theta_2\wedge \theta_3\wedge \theta_4$: $H^{0}$ is one dimensional, generated by $\{1\}$, $H^{1}$ is one dimensional, generated by $\{\theta_1\}$, $H^{3}$ is one dimensional, generated by $\{ \theta_2\wedge \theta_3\wedge \theta_4 \}$, $H^{4}$ is one dimensional, generated by $\{ \theta_1\wedge \theta_2\wedge \theta_3\wedge \theta_4 \}$.

On $G_1$ we let $\theta_1=d\theta$. 

On $G_2$ we use the global exponential coordinate system $\left[ \begin {array}{ccc} 1&x&z+\frac{1}{2}\,xy\\ \noalign{\medskip}0&1&y
\\ \noalign{\medskip}0&0&1\end {array} \right] 
$ of $G_2$ to write $\theta_2=dx,\theta_3=dy,\theta_4=\frac{y}{2}dx-\frac{x}{2}dy+dz$.

Let $\vp=(\theta,0,0)$, $\psi=(0,x+\mathbf{i}y,z)$, we calculate 
\begin{align}
\nu(\vp,\psi)=\vp(\psi \tlt \vp)^{-1}&=\big(\theta,-e^{-\mathbf{i}\theta}(x+\mathbf{i}y),-z \big) \\
&=\big(\theta,(-\cos\theta x - \sin \theta y)+i(- \cos \theta y + \sin \theta x),-z \big) 
\end{align}

Now use~\ref{mumap} we calculate
\begin{align*}
\nu^\ast(\theta_1)=d\theta, \qquad \text {therefore,}\qquad \mu (\theta_1)=1 \otimes \theta_1
\end{align*}
\begin{align*}
\nu^\ast(\theta_2\wedge \theta_3\wedge \theta_4)=&\nu^\ast(dx\wedge dy \wedge dz) \\
=&d (-\cos\theta x - \sin \theta y) \wedge d (- \cos \theta y + \sin \theta x) \wedge dz\\
=&(\cos^2\theta+\sin^2\theta)dx \wedge dy \wedge dz \\
+ &(\sin^2 \theta x -\sin \theta \cos  \theta y+ \cos \theta \sin \theta y + \cos^2 \theta x) d \theta \wedge dx \wedge dz \\
+&(-\sin \theta \cos \theta x + \cos^2  \theta y+ \sin^2 \theta y + \cos \theta  \sin \theta x) d \theta \wedge dy \wedge dz\\
=&dx \wedge dy \wedge dz + x d \theta \wedge dx \wedge dz + y d \theta \wedge dy \wedge dz
\end{align*}
hence 
\begin{align*}
\mu_0 (\theta_2\wedge \theta_3\wedge \theta_4)=& dx \wedge dy \wedge dz \otimes 1 \\
\mu_1(\theta_2\wedge \theta_3\wedge \theta_4)=&(x  dx \wedge dz + y  dy \wedge dz )\otimes \theta_1
\end{align*}
\begin{align*}
\nu^\ast(\theta_1 \wedge \theta_2\wedge \theta_3\wedge \theta_4)=&\nu^\ast(d\theta \wedge dx\wedge dy \wedge dz) \\
=&d\theta  \wedge d (-\cos\theta x - \sin \theta y) \wedge d (- \cos \theta y + \sin \theta x) \wedge dz\\
=&(\cos^2\theta+\sin^2\theta)d\theta \wedge dx \wedge dy \wedge dz 
\end{align*}
hence 
\begin{align*}
\mu (\theta_1 \wedge \theta_2\wedge \theta_3\wedge \theta_4)=d \theta \wedge  dx \wedge dy \wedge dz \otimes 1 \end{align*}

Therefore the images of $\theta_1$ and $\theta_2\wedge \theta_3\wedge \theta_4$ give two even cocycles and the image of $1$,$\theta_1\wedge \theta_2\wedge \theta_3\wedge \theta_4$ give two odd cocycles. We take $\theta_2\wedge \theta_3\wedge \theta_4$ as an example.

\begin{align*} 
\mathcal{E} (\theta_2\wedge \theta_3\wedge \theta_4) (\psi_0,\dots,\psi_{3})=1 \cdot  \int\limits_{\Delta(\psi_0,\dots,\psi_{3})} dx \wedge dy \wedge dz
\end{align*}
where
\begin{align*} 
&dx=(x_1-x_0) dt_1 +(x_2-x_0) dt_2 +(x_3-x_0) dt_3, \\
&dy=(y_1-y_0) dt_1 +(y_2-y_0) dt_2 +(y_3-y_0) dt_3, \\
&dz=(z_1-z_0) dt_1 +(z_2-z_0) dt_2 +(z_3-z_0) dt_3.
\end{align*}
therefore the $D^{3,0}$ part of $\theta_2\wedge \theta_3\wedge \theta_4$ is
\begin{align*} 
\mathcal{E} (\theta_2\wedge \theta_3\wedge \theta_4) (\psi_0,\dots,\psi_{3})=\frac{1}{2}\sum_{\sigma \in S_3} (-1)^{\sigma} (x_{\sigma(1)}-x_0) (y_{\sigma(2)}-y_0) (z_{\sigma(3)}-z_0) 
\end{align*}
Now for the other component we have 
\begin{align*} 
\mathcal{E} (\theta_2\wedge \theta_3\wedge \theta_4) (\psi_0,\dots,\psi_{2})=\theta_1 \cdot  \int\limits_{\Delta(\psi_0,\dots,\psi_{2})} y \cdot dy \wedge dz x \cdot dx \wedge dz 
\end{align*}
where
\begin{align*} 
&dx=(x_1-x_0) dt_1 +(x_2-x_0) dt_2, \\
&dy=(y_1-y_0) dt_1 +(y_2-y_0) dt_2, \\
&dz=(z_1-z_0) dt_1 +(z_2-z_0) dt_2.
\end{align*}
therefore the $D^{2,1}$ part of $\theta_2\wedge \theta_3\wedge \theta_4$ is
\begin{align*} 
&\mathcal{E} (\theta_2\wedge \theta_3\wedge \theta_4) (\psi_0,\dots,\psi_{2})\\
=&\frac{1}{6}\theta_1 \cdot ((y_0^2+x_0^2)(z_1-z_2)+(y_1^2+x_1^2)(z_2-z_0)+(y_2^2+x_2^2)(z_0-z_1))
 \end{align*}

Use $\Phi$ to go direct from these two parts to the representative cochain on algebra $\cA=C_c^{\ify}(G_1)\rtimes G_2^{\delta}$, we have $\Phi(\theta_2 \wedge \theta_3 \wedge \theta_4) \in C^{4}(\cA) \cup C^{2}(\cA)$. 

As an example, the $C^{2}(\cA)$ part is:
\begin{align*}
&\Phi(\theta_2 \wedge \theta_3 \wedge \theta_4) (f_{0} U^{\ast}_{\psi_{0} },f_{1} U^{\ast}_{\psi_{1} },f_{2} U^{\ast}_{\psi_{2} })\\
=&\big((y_0^2+x_0^2)(z_1-z_2)+(y_1^2+x_1^2)(z_2-z_0)+(y_2^2+x_2^2)(z_0-z_1) \big) \int_{\Rb}f_0 f_{1}f_{2} \theta_1
\end{align*}

\subsection{Other examples}
Other examples like higher diamond group can also be calculated explicitly. We can use mathematical software like Maple\texttrademark~to initiate example Lie groups and Lie algebras and calculate all the Lie algebra cohomology classes we want, and then implement $\mathcal{E}$, $\Theta$ and $\Phi$ maps as in the paper.

\section{Appendix}
\subsection{detail of step 1}
We give a detailed account of the transition from $CC^{\bullet}(\cR(G_2) \acl \cU(\Fg_1),^{\sigma^{-1}}\hspace{-3pt} \bC _ \delta)$ to $C^{\bullet,\bullet}(\wedge \Fg_1^{\ast},\otimes \cR(G_2))$ in step 1 in \autoref{subsec:bicomplexes}. We just restate what is proved in~\cite[4.2]{rangipour_van_2012}, under our setting.
\newline
Define a bi-cyclic complex $C^{\bullet,\bullet}(\cU,\cR,^{\sigma^{-1}}\hspace{-3pt}\bC _{\delta})$, where
\begin{equation*}
C^{p,q}(\cU,\cR,^{\sigma^{-1}}\hspace{-3pt} \bC _{\delta}):=^{\sigma^{-1}}\hspace{-3pt} \bC _{\delta}   \otimes \cR^{p } \otimes \cU^{q}, \qquad p,q \ge 0
\end{equation*}

\begin{equation}
\begin{tikzpicture}[description/.style={fill=white,inner sep=2pt}]
\matrix (m) [matrix of math nodes, row sep=3em, column sep=3em, 
text height=1.5ex, text depth=0.25ex]
{   \vdots & \vdots & \vdots    \\
     ^{\sigma^{-1}}\hspace{-3pt} \bC _{\delta}   \otimes  \cU^{\otimes 2}& ^{\sigma^{-1}}\hspace{-3pt} \bC _{\delta}   \otimes  \cU^{\otimes 2} \otimes \cR&  ^{\sigma^{-1}}\hspace{-3pt} \bC _{\delta}   \otimes  \cU^{\otimes 2} \otimes \cR^{\otimes 2}&  \cdots   \\ 
     ^{\sigma^{-1}}\hspace{-3pt} \bC _{\delta}   \otimes  \cU^{} & ^{\sigma^{-1}}\hspace{-3pt} \bC _{\delta}   \otimes  \cU^{} \otimes \cR&  ^{\sigma^{-1}}\hspace{-3pt} \bC _{\delta}   \otimes  \cU^{} \otimes \cR^{\otimes 2}&  \cdots   \\ 
    ^{\sigma^{-1}}\hspace{-3pt} \bC _{\delta}   & ^{\sigma^{-1}}\hspace{-3pt} \bC _{\delta}  \otimes \cR^{}& ^{\sigma^{-1}}\hspace{-3pt} \bC _{\delta}  \otimes \cR^{\otimes 2}&  \cdots   \\  };
  
\path[transform canvas={yshift=0.6ex},->,font=\scriptsize]
(m-2-1) edge node[above] {$ b_{\cR} $} (m-2-2)  (m-2-2) edge node[above] {$ b_{\cR} $} (m-2-3) (m-2-3) edge node[above] {$ b_{\cR} $} (m-2-4)  
(m-3-1) edge node[above] {$ b_{\cR} $} (m-3-2)  (m-3-2) edge node[above] {$ b_{\cR} $} (m-3-3) (m-3-3) edge node[above] {$ b_{\cR} $} (m-3-4)  
(m-4-1) edge  node[above] {$ b_{\cR} $} (m-4-2)  (m-4-2) edge node[above] {$ b_{\cR} $} (m-4-3) (m-4-3) edge node[above] {$ b_{\cR} $} (m-4-4)    ;

\path[transform canvas={xshift=-0.6ex},->,font=\scriptsize]
(m-4-1) edge node[left]{$b_{\cU}$} (m-3-1)  (m-4-2) edge node[left]{$b_{\cU}$} (m-3-2)   (m-4-3) edge node[left]{$b_{\cU}$} (m-3-3)  
(m-3-1) edge node[left]{$b_{\cU}$} (m-2-1)  (m-3-2) edge node[left]{$b_{\cU}$} (m-2-2)   (m-3-3) edge node[left]{$b_{\cU}$} (m-2-3) 
(m-2-1) edge node[left]{$b_{\cU}$} (m-1-1)  (m-2-2) edge node[left]{$b_{\cU}$} (m-1-2)   (m-2-3) edge node[left]{$b_{\cU}$} (m-1-3) ;

\path[transform canvas={yshift=-0.6ex},->,font=\scriptsize]
(m-2-2) edge node[below] {$ B_{\cR} $} (m-2-1)  (m-2-3) edge node[below] {$ B_{\cR} $} (m-2-2) (m-2-4) edge node[below] {$ B_{\cR} $} (m-2-3)  
(m-3-2) edge node[below] {$ B_{\cR} $} (m-3-1)  (m-3-3) edge node[below] {$ B_{\cR} $} (m-3-2) (m-3-4) edge node[below] {$ B_{\cR} $} (m-3-3)  
(m-4-2) edge node[below] {$ B_{\cR} $} (m-4-1)  (m-4-3) edge node[below] {$ B_{\cR} $} (m-4-2) (m-4-4) edge node[below] {$ B_{\cR} $} (m-4-3)    ;

\path[transform canvas={xshift=0.6ex},->,font=\scriptsize]
(m-1-1) edge node[right]{$B_{\cU}$} (m-2-1)  (m-1-2) edge node[right]{$B_{\cU}$} (m-2-2)   (m-1-3) edge node[right]{$B_{\cU}$} (m-2-3)  
(m-2-1) edge node[right]{$B_{\cU}$} (m-3-1)  (m-2-2) edge node[right]{$B_{\cU}$} (m-3-2)   (m-2-3) edge node[right]{$B_{\cU}$} (m-3-3) 
(m-3-1) edge node[right]{$B_{\cU}$} (m-4-1)  (m-3-2) edge node[right]{$B_{\cU}$} (m-4-2)   (m-3-3) edge node[right]{$B_{\cU}$} (m-4-3) ;

\end{tikzpicture} 
\end{equation}

We refer to~\cite[Sec.~3.2]{moscovici_hopf_2009} for the detail of the construction of the above complex and the proof of its bi-cyclicity. 

Next, we identify the diagonal of the above bi-cyclic complex $D^{\bullet}(\cU,\cR,^{\sigma^{-1}}\hspace{-3pt}\bC _{\delta})$ with the standard Hopf cyclic module $CC^{\bullet}(\cR(G_2)   \acl  \cU(\Fg_1) ,^{\sigma^{-1}}\hspace{-3pt}\bC _{\delta})$ by the map, which is similar to the one defined in~\cite[p.~766-767]{moscovici_hopf_2009} and~\cite[p.~502]{rangipour_van_2012}, $\Psi_{\aclsub}:CC^{\bullet}(\cR(G_2)   \acl  \cU(\Fg_1),^{\sigma^{-1}}\hspace{-3pt}\bC _{\delta}) \to D^{\bullet}(\cU,\cR,^{\sigma^{-1}}\hspace{-3pt}\bC _{\delta}) $:

\begin{equation*}
\begin{aligned}
\Psi_{\aclsub}^{}&(1 \otimes F^1 \acl u^1\otimes \cdots \otimes F^n \acl u^n) \\
=&1 \otimes F^1 \otimes F^2 S(u^1_{<n-1>})\otimes F^3 S(u^1_{<n-2>})S(u^2_{<n-2>})\otimes \cdots \\
&\otimes F^nS(u^1_{<1>})S(u^2_{<1>})\cdots S(u^{n-1}_{<1>}) \otimes u^1_{<0>}\otimes \cdots \otimes u^{n-1}_{<0>}\otimes u^{n},
\end{aligned}
\end{equation*}

\begin{equation*}
\begin{aligned}
\Psi_{\aclsub}^{-1}&(1 \otimes F^1\otimes \cdots \otimes F^n\otimes u^1 \otimes \cdots \otimes u^n) \\
&=1 \otimes F^1 \acl u^1_{<0>} \otimes F^2 u^1_{<1>} \acl u^2_{<0>}  \otimes  \cdots \otimes F^n u^1_{<n-1>} \cdots u^{n-1}_{<1>}  \acl u^n.
\end{aligned}
\end{equation*}
We relate $\text{Tot} C^{\bullet,\bullet}(\cU,\cR,^{\sigma^{-1}}\hspace{-3pt}\bC _{\delta})$ with the diagonal $D^{\bullet}(\cU,\cR,^{\sigma^{-1}}\hspace{-3pt}\bC _{\delta})$ by Alexander-Whitney map and shuffle map~(cf.~\cite{khalkhali_generalized_2004}):
\begin{align}
\begin{aligned}
&\qquad AW=\bigoplus_{p+q=n} AW_{p,q},\\
&AW_{p,q}:^{\sigma^{-1}}\hspace{-3pt} \bC _{\delta}   \otimes  \cU^{q} \otimes \cR^{p} \to ^{\sigma^{-1}}\hspace{-3pt} \bC _{\delta}   \otimes  \cU^{p+q} \otimes \cR^{p+q} \\
&AW_{p,q}=(-1)^{p+q} \uparrow\partial_{0}\cdots \uparrow\partial_{0} \overrightarrow{\partial}_{p+q}\cdots \overrightarrow{\partial}_{p+1},\\
&Sh=\sum_{\sigma \in Sh(p,q)} (-1)^{\sigma} \uparrow s_{\sigma(1)}\cdots \uparrow s_{\sigma(p)}\otimes \overrightarrow s_{\sigma(p+1)}\cdots \overrightarrow s_{\sigma(p+q)}.
\end{aligned}
\end{align}
We note that both Alexander-Whitney map and shuffle map are only chain maps between Hochschild complexes. The lack of explicit cyclic Alexander-Whitney map is the reason that the vertical map in diagram~\ref{diagram1} can not be inverted.\\
\newline
Follow the treatment in~\cite[Sec.~4.2]{rangipour_van_2012}, we can go from the complex $C^{p,q}(\cU,\cR,^{\sigma^{-1}}\hspace{-3pt}\bC _{\delta})$ to the quasi-isomorphic complex 
$C^{p,q}(\wedge \Fg_1^{},\otimes \cR(G_2),^{\sigma^{-1}}\hspace{-3pt}\bC _{\delta}):= ^{\sigma^{-1}}\hspace{-3pt}\bC _{\delta} \otimes\wedge^{q} \Fg_1^{}\otimes \cR(G_2)^{\otimes p}$ through antisymmetrization map. We can continue from the complex $C^{\bullet,\bullet}(\wedge \Fg_1^{},\otimes \cR(G_2),^{\sigma^{-1}}\hspace{-3pt}\bC _{\delta})$ to the quasi-isomorphic complex $C^{\bullet,\bullet}(\wedge \Fg_1^{\ast},\otimes \cR(G_2)):= \wedge^{\bullet} \Fg_1^{\ast}\otimes \cR(G_2)^{\otimes \bullet}$ by Poincar\'{e} isomorphism:\\

\begin{align}
\begin{aligned}
\mathfrak{D}&: \wedge^{\text{\tiny dim}(\Fg_1)}  \Fg_1^{\ast} \otimes      \wedge^{\text{\tiny dim}(\Fg_1)-q} \Fg_1            \otimes \cR(G_2)^{\otimes p}   \to  \wedge^{q} \Fg_1^{\ast}\otimes \cR(G_2)^{\otimes p}  \\
\mathfrak{D}& (\varpi \otimes \eta \otimes F^1 \otimes \cdots \otimes F^p) =  \iota ( \eta) \varpi \otimes F^1 \otimes \cdots \otimes F^p,
\end{aligned}
\end{align}
where $\varpi $ is a volume form and $\iota ( \eta) $ is the contraction by $\eta$.

As in~\cite[Sec.~4.2]{rangipour_van_2012} and~\cite[Sec.~1.3]{moscovici_hopf_2011}, we have the right coadjoint action of $\Fg_1$ on $\varpi $
\begin{align}
\text{ad}^{\ast}(Z) \varpi  = \delta(Z) \varpi, \qquad \forall Z \in \Fg_1,
\end{align}
and identify $\wedge^{\text{\tiny dim}(\Fg_1)}  \Fg_1^{\ast} $ with $\bC_{\delta}$ as right $\Fg_1$-modules. 

use~\ref{leftcoaction2} and~\ref{volumeform1} we have the left coaction of $\cR(G_2)$ on $\varpi ^{\ast}$
\begin{align}
\blacktriangledown (\varpi ) = \sigma^{-1} \otimes \varpi ,
\end{align}
and identify $\wedge^{\text{\tiny dim}(\Fg_1)}  \Fg_1^{\ast} $ with $^{\sigma^{-1}}\hspace{-3pt}\bC _{\delta}$ as right $\Fg_1$-module and left $\cR(G_2)$-comodule, or, at the same time, as right-left SAYD over $\cR(G_2)   \acl  \cU(\Fg_1)$.

\endgroup

\bibliographystyle{acm}
\bibliography{EHCCBP}
\nocite{knapp_lie_2002}
\nocite{helgason_differential_1978}
\nocite{moscovici_cyclic_2007}
\nocite{hochschild_cohomology_1962}
\nocite{getzler_cyclic_1993}
\nocite{dixmier_enveloping_1977}
\end{document}